\documentclass[final,leqno]{siamart171218}
\usepackage{amsmath,amsfonts,amssymb}
\usepackage{mathtools}
\usepackage{float}
\usepackage{geometry}
\usepackage{graphicx}
\usepackage{hyperref}
\usepackage{url}
\usepackage{tikz,xypic,subcaption}
\usepackage{cite}
\usepackage{mleftright}

\newcommand\fullwidthdisplay{\displayindent0pt \displaywidth\columnwidth}
\AtBeginDocument{
  \everydisplay\expandafter{\expandafter\fullwidthdisplay\the\everydisplay}
}

\newtheorem{thm}{Theorem}

\newtheorem{lemm}{Lemma}[thm]
\newtheorem{prop}{Proposition}

\newtheorem{assm}{Assumption}

\let\left\mleft
\let\right\mright

\newcommand{\norm}[1]{\left\lVert#1\right\rVert}

\xyoption{rotate}

\begin{document}

\title{Asymptotically compatibility of a class of numerical schemes for a nonlocal traffic flow model \thanks{This rsearch is supported in part by US NSF DMS-1937254, DMS-2012562, and CNS-2038984.
}}


\author{Kuang Huang\thanks{Department of Applied Physics and Applied Mathematics, 
 Columbia University, New York, NY 10027; {\tt kh2862@columbia.edu}} \and Qiang Du\thanks{Department of Applied Physics and Applied Mathematics,  and Data Science Institute, Columbia University, New York, NY 10027; {\tt qd2125@columbia.edu}}}

\maketitle

\begin{abstract}

This paper considers numerical discretization of a nonlocal conservation law modeling vehicular traffic flows involving nonlocal inter-vehicle interactions.
The nonlocal model involves an integral over the range measured by a horizon parameter and it recovers the local Lighthill-Richards-Whitham model as the nonlocal horizon parameter goes to zero. Good numerical schemes for simulating these parameterized nonlocal traffic flow models should be robust with respect to the change of the model parameters but this has not been systematically investigated in the literature. We fill this gap through a careful study of a class of finite volume numerical schemes with suitable discretizations of the nonlocal integral, which include several schemes proposed in the literature and their variants. Our main contributions are to demonstrate the asymptotically compatibility of the schemes, which includes both the uniform convergence of the numerical solutions to the unique solution of nonlocal continuum model for a given positive horizon parameter and the convergence to the unique entropy solution of the local model as the mesh size and the nonlocal horizon parameter go to zero simultaneously. It is shown that with the asymptotically compatibility, the schemes can provide robust numerical computation under the changes of the nonlocal horizon parameter.
\end{abstract}



\section{Introduction}
In this work, 
we study the numerical discretization of a nonlocal analog of the classical Lighthill-Richards-Whitham (LWR) model \cite{lighthill1955kinematic,richards1956shock}. The latter, given by 
\begin{align}
	\partial_t \rho(t,x) + \partial_x \left( \rho(t,x) v(\rho(t,x)) \right) =0 , \label{eq:lwr}
\end{align}
for a density $\rho=\rho(t,x)$ and a velocity $v= v( \rho (t, x))$,
has been widely used in the study of traffic flows. 
To study the dynamics of traffic flows in the presence of nonlocal inter-vehicle interactions \cite{Blandin2016,goatin2016well},
the following \emph{nonlocal LWR model} has been developed in recent years
\begin{align}
	\partial_t \rho(t,x) + \partial_x \left( \rho(t,x) v_\delta( \rho(t,\cdot), t, x)
 \right) = 0, \quad x\in\mathbb{R}, \, t>0. \label{eq:nonlocal_lwr}
\end{align}
In contrast to \eqref{eq:lwr}, the nonlocal LWR model \eqref{eq:nonlocal_lwr} adopts a modeling assumption that in a fleet of vehicles driving on a highway, each vehicle decides its driving speed not by the local information but rather through a nonlocal weighted average of traffic information within a road segment of length $\delta>0$ ahead of the vehicle's current location. More specifically, 
the velocity $v_\delta= v_\delta( \rho, t, x)$ 
takes on the form
\begin{align}
\label{eq:nonlocal_velocity}
  v_\delta( \rho(t,\cdot), t, x)
  =v(q_\delta( \rho(t,\cdot), t, x ) ), \quad\text{with}\quad
  q_\delta( \rho(t,\cdot), t, x )
 = \int_0^\delta \rho(t,x+s) w_\delta(s) \,ds,
 \end{align}
 where the integral kernel $w=w_\delta(s)$ is assumed to be a probability density function defined on the interval $[0,\delta]$. 
Alternatively, one may also consider the nonlocal velocity given by \cite{friedrich2022conservation}
\begin{align}
\label{eq:nonlocal_velocity-a}
v_\delta( \rho(t,\cdot), t, x)
  = \int_0^\delta  v(\rho(t,x+s) ) w_\delta(s) \,ds.
 \end{align}
The equation \eqref{eq:nonlocal_lwr} is solved with the initial condition:
\begin{align}\label{eq:ini_data}
	\rho(0,x) = \rho_0(x), \quad x\in\mathbb{R}, 
\end{align}
where $\rho_0: \, \mathbb{R}\to[0,1]$ represents the initial traffic density.
The case $\rho_0 \equiv 0$ indicates that the road is empty and the case $\rho_0 \equiv 1$ corresponds to fully congested traffic.

The equation \eqref{eq:nonlocal_lwr} leads to a nonlocal conservation law due to the nonlocal dependence of the velocity on the density.
Consider the rescaled kernel $w_\delta(s)=w(s/\delta)/\delta$ such that $w_\delta$ converges to a Dirac point mass as $\delta\to0$, it is clear that the nonlocal LWR model \eqref{eq:nonlocal_lwr}, with either choices of the velocity given by \eqref{eq:nonlocal_velocity} or \eqref{eq:nonlocal_velocity-a},
formally recovers the local model \eqref{eq:lwr} by taking the limit $\delta\to 0$.
For more rigorous analysis of the nonlocal LWR model \eqref{eq:nonlocal_lwr}, we refer to a number of existing studies in the literature, including the model well-posedness \cite{Blandin2016,goatin2016well,bressan2019traffic,colombo2021local,friedrich2022conservation}, traveling wave solutions \cite{ridder2019traveling,shen2018traveling}, the asymptotic stability of uniform flows \cite{huang2022stability}, and nonlocal-to-local limit as $\delta\to0$ \cite{bressan2019traffic,bressan2020entropy,colombo2021local,colombo2022nonlocal,keimer2019approximation,coclite2020general,colombo2019singular,friedrich2022conservation,keimer2022singular}. 

The numerical discretization of the nonlocal LWR model \eqref{eq:nonlocal_lwr} has also been studied in \cite{Blandin2016,goatin2016well,friedrich2018godunov,friedrich2019maximum,Chalons2018,colombo2021role}.
However, there was no systematic study on the dependence of numerical solutions on the parameter $\delta$ and their behavior under the limit $\delta\to0$.
In the present work, we aim to fill this gap by designing and analysing finite volume numerical schemes for the nonlocal LWR model \eqref{eq:nonlocal_lwr} such that they are able to correctly resolve both the nonlocal model for a given $\delta>0$ and  also the local model \eqref{eq:lwr} when $\delta\to0$.
Such schemes are in the spirit of \emph{asymptotically compatible} schemes, which can offer robust numerical computation under the changes of $\delta$; see \cite{tian2014asymptotically,tian2020asymptotically} for discussions on asymptotically compatibility of numerical discretizations of more general nonlocal models. 
The main contributions of our work here are the rigorous proofs of the asymptotically compatibility of the schemes, which include both the uniform convergence of the numerical solutions to the unique solution of nonlocal continuum model for a given positive horizon parameter and the convergence to the unique entropy solution of the local model as the mesh size and the nonlocal horizon parameter go to zero simultaneously. These results are established for the first time in the literature.
The main ingredients of the proofs are the compactness in the $\mathbf{BV}_{\mathrm{loc}}$ space and the entropy admissibility of numerical solutions.
The analysis provided in \cite{goatin2016well,Blandin2016} was based on a priori $\mathbf{L}^\infty$ and total variation estimates for a fixed $\delta>0$, but the resulting total variation bound blows up to infinity as $\delta\to0$.
In this work, a novelty is our use of a different approach to prove that numerical solutions produced by the proposed schemes satisfy an one-sided Lipschitz condition when $\delta$ is close to zero, which enforces both the boundedness of total variation and the entropy admissibility. Such an approach has been used to study numerical schemes for the local model \eqref{eq:lwr}, see \cite{tadmor1984large,brenier1988discrete}, but to our best knowledge, has not been used for nonlocal models.
Numerical experiments are also reported to complement the theoretical investigation.  Note that while the current work is motivated by modeling traffic flows with nonlocal vehicle interactions,  
let us mention that conservation laws with nonlocal fluxes were also studied in the modeling of pedestrian traffic \cite{Colombo2012,burger2020non}, sedimentation \cite{Betancourt2011}, and material flow on conveyor belts \cite{Goettlich2014,rossi2020well}; see \cite{Aggarwal2015,Amorim2015,Colombo2018,Goatin2019,Chiarello2019,Berthelin2019,karafyllis2020analysis,friedrich2022lyapunov} for more relevant studies. 
Thus, our study here can be useful in the numerical simulations of a broad range of problems in various application domains.

To summarize the paper, in the remainder of this Section 1, after briefly describing the assumptions on the nonlocal model and some basic mathematical properties, we introduce the numerical discretization schemes and summarize the main theorems on their convergence behavior and the asymptotic compatibility. The detailed proofs of the main theorems are given in Section 2. We present results of some numerical experiments in Section 3 and offer some conclusions in Section 4.

\subsection{A review of well-posedness and nonlocal-to-local limit}
Let us first state some assumptions on the model.
\begin{assm}\label{assm:1}

(i) The nonlocal kernel is given by $w_\delta(s)=w(s/\delta)/\delta$ for $s\in[0,\delta]$, where $w=w(s)$ is a $\mathbf{C}^1$ smooth, strictly decreasing, and nonnegative  probability density function defined on $[0,1]$, and it satisfies the normalization condition $\int_0^1 w(s)\,ds=1$.\\
(ii) The velocity function is $v(\rho)=1-\rho$. Consequently,  \eqref{eq:nonlocal_velocity} and \eqref{eq:nonlocal_velocity-a} produce the same outcome.\\
(iii) The initial data $\rho_0 \in \mathbf{L}^\infty(\mathbb{R})$ and it satisfies $0\leq\rho_0(x)\leq1$ for all $x\in\mathbb{R}$. In addition, $\rho_0$ has bounded total variation.
\end{assm}

Concerning the mathematical analysis of the nonlocal LWR model \eqref{eq:nonlocal_lwr}, we recall that 
the existence and uniqueness of weak solutions have been shown with general choices of the nonlocal kernel, the velocity function, and the initial data, see for example, 
\cite{Blandin2016,goatin2016well,bressan2019traffic,colombo2021local}.  
For our case, the following proposition summarizes the known results in the above works.
\begin{prop}\label{prop:nonlocal_sol}
	Under Assumption~\ref{assm:1}, the nonlocal LWR model \eqref{eq:nonlocal_lwr} admits a unique weak solution $\rho \in \mathbf{L}^\infty\left([0,+\infty)\times\mathbb{R}\right)$ such that
	\begin{align}
\int_0^\infty\int_{\mathbb{R}}\rho(t,x)\partial_t\phi(t,x)+\rho(t,x)v_\delta
(\rho(t,\cdot), t, x)
)\partial_x\phi(t,x)\,dxdt+\int_{\mathbb{R}}\rho_0(x)\phi(0,x)\,dx=0,\label{eq:nonlocal_sol}
	\end{align}
	for all $\phi\in\mathbf{C}^1_{\mathrm{c}}\left([0,+\infty)\times\mathbb{R}\right)$, where $v_\delta
(\rho(t,\cdot), t, x)$ is given by \eqref{eq:nonlocal_velocity}. Moreover, the solution satisfies the maximum principle	\begin{align}\label{eq:maxm_principle}
	\inf_{x\in\mathbb{R}}\rho_0(x) \leq \rho(t,x) \leq \sup_{x\in\mathbb{R}}\rho_0(x), \quad (t,x)\in [0,+\infty)\times\mathbb{R}.
	\end{align}
\end{prop}
The convergence of solutions of the nonlocal LWR model \eqref{eq:nonlocal_lwr} as $\delta\to0$ has also been extensively studied. In the literature, it was usually assumed that the nonlocal kernel $w=w_\delta(s)$ is defined for $s\in[0,+\infty)$ and the nonlocal density is defined by
\begin{align*}
	q_\delta(\rho(t,\cdot),t,x) = \int_0^\infty \rho(t,x+s) w_\delta(s) \,ds.
\end{align*}
\cite{bressan2019traffic,bressan2020entropy} considered the exponential kernels $w_\delta(s)=\delta^{-1}e^{-\frac{s}{\delta}}$ and showed convergence from the solutions of the nonlocal model \eqref{eq:nonlocal_lwr} to the unique weak entropy solution of the local model \eqref{eq:lwr}, assuming that the initial data $\rho_0$ is uniformly positive. \cite{colombo2021local} generalized the convergence result for a class of nonlocal kernels with exponential decay rate but under one additional assumption that $\rho_0$ is one-sided Lipschitz continuous. In \cite{colombo2021local}, the authors also provided counterexamples to show that the uniform positivity of the initial data is essential to the convergence result.
In the subsequent works \cite{coclite2020general,colombo2022nonlocal,keimer2019approximation,friedrich2022conservation,keimer2022singular}, convergence results concerning the nonlocal quantity $q_\delta(\rho(t,\cdot),t,x)$ as $\delta\to0$ were given without assuming the initial data being uniformly positive.

In the present work, we adopt an approach similar as that in \cite{colombo2021local} and make the following additional assumption on the initial data, which basically requires the initial data to be uniformly positive and to have no negative jumps.  
\begin{assm}\label{assm:2}
The initial data $\rho_0$ satisfies
\begin{align}
	\rho_0(x) \geq \rho_{\mathrm{min}} > 0 \quad \forall x\in\mathbb{R}, \qquad -\frac{\rho_0(y)-\rho_0(x)}{y-x}\leq L \quad \forall x\neq y\in\mathbb{R}, \label{eq:ini_lip_const}
\end{align}
for some constants $\rho_{\mathrm{min}}>0$ and $L>0$.
\end{assm}

In our case, the same arguments as in \cite{colombo2021local} can be applied to give the nonlocal-to-local limit result, as stated in the following Proposition~\ref{prop:local_limit}, with very little modifications for compactly supported nonlocal kernels.

\begin{prop}\label{prop:local_limit}
	Suppose Assumptions~\ref{assm:1} and \ref{assm:2} are satisfied.
	As $\delta\to0$, the solution of the nonlocal LWR model \eqref{eq:nonlocal_lwr} converges in $\mathbf{L}^1_{\mathrm{loc}}([0,+\infty)\times\mathbb{R})$ to the weak entropy solution of the local model \eqref{eq:lwr} that satisfies
	\begin{align}
\int_0^\infty\int_{\mathbb{R}}\rho(t,x)\partial_t\phi(t,x)+\rho(t,x)v\left(\rho(t,x)\right)\partial_x\phi(t,x)\,dxdt+\int_{\mathbb{R}}\rho_0(x)\phi(0,x)\,dx=0,\label{eq:local_sol}
	\end{align}
	for all $\phi\in\mathbf{C}^1_{\mathrm{c}}\left([0,+\infty)\times\mathbb{R}\right)$, and
	\begin{align}
		-\frac{\rho(t,y)-\rho(t,x)}{y-x}\leq \frac{1}{2t} \quad \forall x\neq y\in\mathbb{R},\,t>0.\label{eq:oleinik}
	\end{align}
\end{prop}

In Proposition~\ref{prop:local_limit}, the inequality \eqref{eq:oleinik}, which is known as the Oleinik's entropy condition, is used to select the unique entropy admissible solution of the scalar conservation law \eqref{eq:lwr}, see \cite{lefloch2002hyperbolic}.
As a constraint on the one-sided Lipschitz constant of the solution, the entropy condition \eqref{eq:oleinik} yields that the solution can only have positive jumps.

\subsection{Finite volume approximations}
Now let us consider the numerical discretization of the nonlocal LWR model \eqref{eq:nonlocal_lwr}. With finite volume approximations, the numerical solution is defined as a piecewise constant function:
\begin{align}\label{eq:num_sol}
\rho(t,x)=\sum_{j\in\mathbb{Z}}\sum_{n=0}^\infty\rho_j^n\mathbf{1}_{\mathcal{C}_j\times\mathcal{T}^n}(t,x),
\end{align}
where $\mathcal{C}_j=(x_{j-1/2},x_{j+1/2})$, $\mathcal{T}^n=(t^n,t^{n+1})$ are spatial and temporal cells. The grid points are $x_j=jh$ and $t^n=n\tau$, where $h$ and $\tau$ are spatial and temporal mesh sizes.
At the initial time $t^0=0$, the initial data is discretized as:
\begin{align}\label{eq:ini_data_discrete}
    \rho_j^0 = \frac1h \int_{\mathcal{C}_j} \rho_0(x) \,dx, \quad j\in\mathbb{Z} .
\end{align}
Denote $F_{j-1/2}^n$ and $F_{j+1/2}^n$ the numerical fluxes across cell boundaries $x_{j-1/2}$ and $x_{j+1/2}$ during time $t^n$ to $t^{n+1}$. Specifying appropriate boundary fluxes, the finite volume scheme is:
\begin{align}
\label{eq:finite_volume}
	\rho_j^{n+1}=\rho_j^n+
 \lambda
 (F_{j-1/2}^n-F_{j+1/2}^n),
\end{align}
where the CFL ratio $\lambda = \tau/h$ is taken to be a fixed constant.
To specify the numerical fluxes, we need to evaluate the nonlocal density $q_\delta(\rho(t,\cdot),t,x)$ given in \eqref{eq:nonlocal_velocity}.
Let us take
\begin{align}\label{eq:discrete_nonlocal_density}
	q_j^n = \sum_{k=0}^{m-1} w_k \rho_{j+k}^n,
\end{align}
where $m=\lceil\frac{\delta}{h}\rceil$ is the number of cells involved in the nonlocal integral, and $\{w_k\}_{k=0}^{m-1}$ is a set of numerical quadrature weights, such that:
\begin{align}
	w_{\delta,h}(s)=\sum_{k=0}^{m-1}w_k\mathbf{1}_{[kh, (k+1)h]}(s), \quad s\in[0,\delta],
\end{align}
is a piecewise constant approximation of the nonlocal kernel $w_\delta=w_\delta(s)$.

Given the discretized nonlocal densities $\{q_j^n\}_{j\in\mathbb{Z}}^{n\geq0}$, the nonlocal fluxes in \eqref{eq:finite_volume} can be constructed in a number of different ways. Let us mention the following examples.
\begin{itemize}
\item In \cite{Blandin2016,goatin2016well}, a Lax-Friedrichs type scheme was developed with the numerical fluxes:
\begin{align}\label{eq:goatin_flux}
	F_{j-1/2}^n=\frac12 \left[\rho_{j-1}^n v\left(\sum_{k=0}^{m-1}w_k\rho_{j+k-1}^n\right) + \rho_j^n v\left(\sum_{k=0}^{m-1}w_k\rho_{j+k}^n\right) \right] + \frac{\alpha}{2}(\rho_{j-1}^n-\rho_j^n),
\end{align}
where $\alpha>0$ is a numerical viscosity constant and the numerical quadrature weights are given by the left endpoint values: 
\begin{align}
	\mbox{[Left endpoint]} \quad & w_k = w_\delta(kh)h, \quad k=0,\cdots,m-1. \label{eq:quad_weight_left} 
\end{align}
\item In \cite{friedrich2018godunov}, a Godunov type scheme was proposed with the numerical fluxes defined by:
\begin{align}\label{eq:upwind_flux}
	F_{j-1/2}^n = \rho_{j-1}^n v\left(\sum_{k=0}^{m-1}w_k\rho_{j+k}^n\right),
\end{align}
where the numerical quadrature weights are given by the exact quadrature:
\begin{align}
	\mbox{[Exact quadrature]} \quad  w_k = \int_{kh}^{\min\{(k+1)h,\delta\}} w_\delta(s)\,ds, \quad k=0,\cdots,m-1. \label{eq:quad_weight_exact} 
\end{align}
\item Inspired by both \eqref{eq:goatin_flux} and \eqref{eq:upwind_flux}, we also consider the following Lax-Friedrichs type fluxes:
\begin{align}\label{eq:lf_flux_new}
	F_{j-1/2}^n=\frac12 \left(\rho_{j-1}^n + \rho_j^n \right) v\left(\sum_{k=0}^{m-1}w_k\rho_{j+k}^n\right) + \frac{\alpha}{2}(\rho_{j-1}^n-\rho_j^n),
\end{align}
where the numerical quadrature weights are given by either the left endpoint values or the exact quadrature.
\end{itemize}

In the present work, we consider a family of finite volume schemes:
\begin{align}\label{eq:nonlocal_lwr_num}
	\rho_j^{n+1} &= \mathcal{H} \left( \rho_{j-1}^n,\rho_j^n,\rho_{j+1}^n,\cdots,\rho_{j+m}^n \right)=\rho_j^n+
 \lambda
 (F_{j-1/2}^n-F_{j+1/2}^n)\\
	&= \rho_j^n + \lambda \left[ g(\rho_{j-1}^n,\rho_j^n,q_{j-1}^n,q_j^n) - g(\rho_j^n,\rho_{j+1}^n,q_j^n,q_{j+1}^n) \right], \label{eq:nonlocal_lwr_num_2}
\end{align}
where $q_j^n$ is given by \eqref{eq:discrete_nonlocal_density}, $\lambda=\tau/h$ is the CFL ratio, and $g=g(\rho_L,\rho_R,q_L,q_R)$ is a \emph{numerical flux function} that depends on both local densities $\rho_L,\rho_R$ and nonlocal densities $q_L,q_R$.
We remark that, by taking $q_L=\rho_L$ and $q_R=\rho_R$, $g=g(\rho_L,\rho_R,\rho_L,\rho_R)$ becomes a numerical flux function for the local model \eqref{eq:lwr}, and the respective numerical scheme:
\begin{align}\label{eq:local_lwr_num}
	\rho_j^{n+1} &= \rho_j^n + \lambda \left[ g(\rho_{j-1}^n,\rho_j^n,\rho_{j-1}^n,\rho_j^n) - g(\rho_j^n,\rho_{j+1}^n,\rho_j^n,\rho_{j+1}^n) \right],
\end{align}
can be viewed as the local counterpart of \eqref{eq:nonlocal_lwr_num}-\eqref{eq:nonlocal_lwr_num_2}.

It is worthwhile to mention that the aforementioned schemes, with numerical fluxes in \eqref{eq:goatin_flux}, \eqref{eq:upwind_flux}
and \eqref{eq:lf_flux_new} respectively, all belong to the above family \eqref{eq:nonlocal_lwr_num}-\eqref{eq:nonlocal_lwr_num_2}, with the numerical flux functions given by:
\begin{subequations}
\begin{align}
	\mbox{[Lax-Friedrichs]} \quad &g(\rho_L, \rho_R, q_L, q_R) = \frac12( \rho_L v(q_L) + \rho_R v(q_R) ) + \frac\alpha2 (\rho_L - \rho_R), \label{eq:g_func_LxF} \\
	\mbox{[Godunov]} \quad &g(\rho_L, \rho_R, q_L, q_R) = \rho_L v(q_R), \label{eq:g_func_Godunov} \\
	\mbox{[modified Lax-Friedrichs]} \quad &g(\rho_L, \rho_R, q_L, q_R) = \frac12( \rho_L + \rho_R ) v(q_R) + \frac\alpha2 (\rho_L - \rho_R), \label{eq:g_func_LxF_new}
\end{align}
\end{subequations}
respectively. Now we make the following assumptions on the numerical quadrature weights, the numerical flux function, and 
the CFL ratio $\lambda$.

\begin{assm}\label{assm:3}
The numerical quadrature weights $\{w_k\}_{0\leq k\leq m-1}$ satisfy
\begin{align}\label{eq:weight_monotone}
    w_\delta(kh)h \geq w_k \geq w_\delta((k+1)h)h \quad \mathrm{and} \quad w_k - w_{k+1} \geq c m^{-2},
\end{align}
for some constant $c>0$ only depending on the kernel function $w=w(s)$.
Moreover, $\{w_k\}_{0\leq k\leq m-1}$ satisfy the normalization condition:
\begin{align}\label{eq:normalization_condition}
    \sum_{k=0}^{m-1} w_k = 1.
\end{align}
\end{assm}

\begin{assm}\label{assm:4}
(i) The numerical flux function $g$ is a quadratic function.
(ii) When $\rho_L=\rho_R$ and $q_L=q_R$, $g(\rho_L,\rho_L,q_L,q_L)=\rho_L(1-q_L)$.
(iii) Denote $\gamma_{ij}$, $1\leq i,j \leq4$ the second order partial derivatives of $g$, they satisfy
\begin{align}
    &\gamma_{11}=\gamma_{12}=\gamma_{22}=0, \quad \gamma_{33}=\gamma_{34}=\gamma_{44}=0, \\    &\gamma_{13},\gamma_{23},\gamma_{14},\gamma_{24}\leq0, \quad \gamma_{13}+\gamma_{23}+\gamma_{14}+\gamma_{24}=-1.
\end{align}
(iv) Denote $\theta^{(i)}$, $1\leq i \leq4$ the first order partial derivatives of $g$ with respect to its four arguments $\rho_L,\rho_R,q_L,q_R$. 
For any $0\leq \rho_L,\rho_R,q_L,q_R\leq1$:
\begin{align}
    &\theta^{(1)}(q_L,q_R)\geq0, \ \theta^{(2)}(q_L,q_R)\leq0, \ \theta^{(3)}(\rho_L,\rho_R)\leq0, \ \theta^{(4)}(\rho_L,\rho_R)\leq0, \\
     &\theta^{(1)}(q_L,q_R) + \theta^{(3)}(\rho_L,\rho_R) + 2(\gamma_{13}+\gamma_{23}) \geq 0, \ \theta^{(2)}(q_L,q_R) - 2(\gamma_{23}+\gamma_{24}) \leq 0, \\
    &\theta^{(3)}(\rho_L,\rho_R) + \theta^{(4)}(\rho_L,\rho_R) \leq -\min\{\rho_L,\rho_R\}.
\end{align}
\end{assm}

\begin{assm}\label{assm:5}
Given the notation $\theta^{(i)}$, $1 \leq i \leq 4$ in Assumption~\ref{assm:4}, $\lambda$
 satisfies 
 \begin{align}
     \lambda \sum_{i=1}^4 \norm{\theta^{(i)}}_\infty < 1.
 \end{align}
\end{assm}

\subsection{Main results}\label{sec:main_result}
This section summarizes the main results. We note that all the theorems are subject to Assumptions~\ref{assm:1}-\ref{assm:5}. To clarify the notation, we denote:
\begin{itemize}
\item $\rho^\delta$: the continuum solution of the nonlocal LWR model \eqref{eq:nonlocal_lwr};
\item  $\rho^0$: the continuum solution of the local LWR model \eqref{eq:lwr};
\item  $\rho^{\delta,h}$: the numerical solution of the nonlocal LWR model; and
\item  $\rho^{0,h}$: the numerical solution of the local LWR model. 
\end{itemize}
There are two sets of parameters: the nonlocal horizon parameter $\delta$ and the mesh size parameter $h$.
In the present work, we are interested in establishing relations between those solutions when $\delta\to0$ and $h\to0$ along various limiting paths, as shown in Fig.~\ref{fig:main_diagram}.

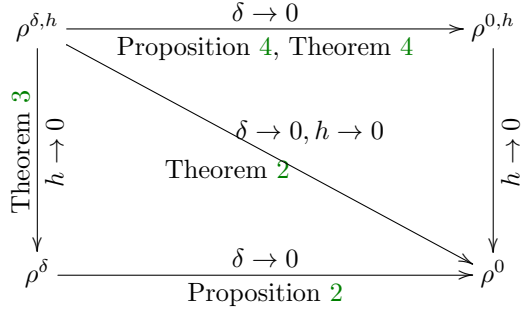
\begin{figure}[htbp]
	\centering	\centerline{\xymatrix@R+4.4pc@C+10.5pc{
    \rho^{\delta, h} \ar[r]^{\mbox{$\delta\to0$}}_{\mbox{Proposition~\ref{prop:scheme_local_limit},\ Theorem~\ref{thm:ap}}} \ar[dr]^{\mbox{$\delta\to0,h\to0$}}_{\mbox{Theorem~\ref{thm:ac}}} \ar[d]^{\rotatebox{90}{$h\to0$}}_{\rotatebox{90}{Theorem~\ref{thm:numerical_convergence}}} & {\rho^{0,h}} \ar[d]^{\rotatebox{90}{$h\to0$}} \\ {\rho^\delta} \ar[r]_{\mbox{Proposition~\ref{prop:local_limit}}}^{\mbox{$\delta\to0$}} & {\rho^0}
    }}
    \vspace{-0.2cm}
    \caption{Diagram of various limiting paths}
    \label{fig:main_diagram}
\end{figure}

\begin{enumerate}
	\item The numerical convergence for the nonlocal model: $\rho^{\delta,h}\to\rho^{\delta}$ when $h\to0$ with fixed $\delta>0$ can be proved following the approach in \cite{goatin2016well}. The proof is based on a priori $\mathbf{L}^\infty$ and total variation estimates of the numerical solution. In Theorem~\ref{thm:numerical_convergence}, we provide a stronger result stating uniform numerical convergence with respect to $\delta$.
	\item The numerical convergence for the local model: $\rho^{0,h}\to\rho^0$ when $h\to0$ is a classical result, see for example \cite{leveque2002finite}.
	\item The nonlocal-to-local limit: $\rho^{\delta}\to\rho^0$ when $\delta\to0$ is given in Proposition~\ref{prop:local_limit}.
	\item The nonlocal-to-local limit of numerical discretizations: $\rho^{\delta,h}\to\rho^{0,h}$ as $\delta\to0$ with fixed $h$ follows from Proposition~\ref{prop:scheme_local_limit}. We also provide a uniform convergence result in Theorem~\ref{thm:ap}.
\end{enumerate}
To complete the convergence diagram in Fig.~\ref{fig:main_diagram}, one would ask whether $\rho^{\delta,h}\to\rho^0$ when both $\delta\to0$ and $h\to0$ simultaneously. If that is the case, we say that the numerical scheme \eqref{eq:nonlocal_lwr_num}-\eqref{eq:nonlocal_lwr_num_2} is \emph{asymptotically compatible} \cite{tian2014asymptotically,tian2020asymptotically} with its local limit.

Our key contribution is to prove the asymptotically compatibility of the proposed scheme \eqref{eq:nonlocal_lwr_num}-\eqref{eq:nonlocal_lwr_num_2}, which is given in Theorem~\ref{thm:ac}.
The proof is base on the a priori $\mathbf{L}^\infty$ and total variation estimates given in Theorem~\ref{thm:a_priori_estimates}, which are uniform to the nonlocal horizon parameter $\delta$.

\begin{thm}\label{thm:a_priori_estimates}
Under Assumptions~\ref{assm:1}-\ref{assm:5}, 
and that
\begin{align}\label{eq:delta_condition}
	0< \delta \leq \delta_0 \doteq \frac{c\rho_{\mathrm{min}}}{2Lw(0)},
\end{align}
where the constant $c$ is as in \eqref{eq:weight_monotone} and the constants $\rho_{\mathrm{min}}$ and $L$ are as in \eqref{eq:ini_lip_const}.
The numerical solution $\rho^{\delta,h}$ produced by the scheme \eqref{eq:nonlocal_lwr_num}-\eqref{eq:nonlocal_lwr_num_2} satisfies the maximum principle
\begin{align}\label{eq:numerical_maxm_principle}
	\inf_{x\in\mathbb{R}}\rho_0(x) \leq \rho^{\delta,h}(t,x) \leq \sup_{x\in\mathbb{R}}\rho_0(x), \quad (t,x)\in [0,+\infty)\times\mathbb{R}.
\end{align}
Moreover, the total variation of the numerical solution in space $\mathrm{TV}(\rho^{\delta,h}(t,\cdot))$ is a non-increasing function of $t\in[0,+\infty)$, and
\begin{align}\label{eq:numerical_tvd}
	\mathrm{TV}(\rho^{\delta,h}; \, [0,T]\times\mathbb{R}) \leq T\cdot\mathrm{TV}(\rho_0) \quad \forall T>0.
\end{align}
\end{thm}

\begin{thm}\label{thm:ac}
Under Assumptions~\ref{assm:1}-\ref{assm:5},
when $\delta\to0$ and $h\to0$ simultaneously,
the numerical solution $\rho^{\delta,h}$ produced by the scheme \eqref{eq:nonlocal_lwr_num}-\eqref{eq:nonlocal_lwr_num_2}  converges in $\mathbf{L}^1_{\mathrm{loc}}([0,+\infty)\times\mathbb{R})$ to the weak entropy solution $\rho^0$ of the local model \eqref{eq:lwr} as defined in Proposition~\ref{prop:local_limit}.
\end{thm}

Based on the asymptotically compatibility of the scheme, we can show numerical convergence from $\rho^{\delta,h}$ to $\rho^\delta$ uniformly in $\delta$, which guarantees robustness of numerical computation when using the scheme \eqref{eq:nonlocal_lwr_num}-\eqref{eq:nonlocal_lwr_num_2} under changes to $\delta$. Moreover, we can also give uniform convergence from $\rho^{\delta,h}$ to $\rho^{0,h}$ with respect to the mesh size $h$. Such a property is referred to as \emph{asymptotic preserving} in the literature \cite{jin2010asymptotic,filbet2010class,jin1999efficient}.

\begin{thm}\label{thm:numerical_convergence}
Under Assumptions~\ref{assm:1}-\ref{assm:5},
and that $\delta$ satisfies the condition \eqref{eq:delta_condition},
as $h\to0$,
the numerical solution $\rho^{\delta,h}$ produced by the scheme \eqref{eq:nonlocal_lwr_num}-\eqref{eq:nonlocal_lwr_num_2} converges in $\mathbf{L}^1_{\mathrm{loc}}([0,+\infty)\times\mathbb{R})$ to the weak solution $\rho^{\delta}$ of the nonlocal model \eqref{eq:nonlocal_lwr} as defined in Proposition~\ref{prop:nonlocal_sol}.
Moreover, the convergence is uniform with respect to $\delta\in(0,\delta_0]$ where $\delta_0$ is as in \eqref{eq:delta_condition}:
\begin{align}\label{eq:uniform_numerical_convergence}
	\lim_{h\to0} \left. \sup_{\delta\in(0,\delta_0]} \norm{\rho^{\delta,h} - \rho^{\delta}}_{\mathbf{L}^1(U)} \right. = 0 \quad \mathrm{for\ any\ bounded\ } U\subset [0,+\infty)\times\mathbb{R}.
\end{align}
\end{thm}

Let us make some remarks on the convergence rates in the above Theorem~\ref{thm:ac} and Theorem~\ref{thm:numerical_convergence}.
On one hand, the scheme \eqref{eq:nonlocal_lwr_num}-\eqref{eq:nonlocal_lwr_num_2} is expected to be at most first order accurate because it is based on a piecewise constant approximation.
On the other hand, for scalar conservation laws, it is known that a first order monotone scheme may have a $O(h^{1/2})$ convergence rate for discontinuous solutions \cite{leveque2002finite}.
In the numerical experiments in Section~\ref{sec:numerical_experiments}, we test the scheme with both smooth initial data and discontinuous ones, the results validate the $O(h)$ convergence rate to the local solution (as in Theorem~\ref{thm:ac}) when $\delta=mh$ for a fixed integer $m>0$, and the $O(h)$ convergence rate to the nonlocal solution uniformly in $\delta$ (as in Theorem~\ref{thm:numerical_convergence}).
We leave the rigorous analysis of convergence rates along various limiting paths in the future works.

Finally, we can also obtain the nonlocal-to-local limit of numerical discretizations, in particular, the following uniform convergence result.

\begin{thm}\label{thm:ap}
Under Assumptions~\ref{assm:1}-\ref{assm:5},
	for any $h_0>0$, we have
	\begin{align}\label{eq:ap_conv}
      \lim_{\delta\to0} \left. \sup_{h\in(0,h_0]} \norm{\rho^{\delta,h} - \rho^{0,h}}_{\mathbf{L}^1(U)} \right. = 0 \quad \mathrm{for\ any\ bounded\ } U\subset [0,+\infty)\times\mathbb{R}.
	\end{align}
\end{thm}

\subsection{Comments on numerical quadrature weights and numerical flux functions}
Let us make some remarks on the choice of the numerical quadrature weights $\{w_k\}_{0\leq k\leq m-1}$. Provided that the nonlocal kernel $w_\delta=w_\delta(s)$ is $\mathbf{C}^1$ smooth and decreasing, one can write the numerical quadrature weights as
\begin{align*}
	w_k = w(\xi_k)\frac{h}{\delta}
 , \quad \xi_k \in \left[ k\frac{h}{\delta}, (k+1)\frac{h}{\delta} \right], \quad k=0,\cdots,m-1,
\end{align*}
where $\{\xi_k\}_{0\leq k\leq m-1}$ can be viewed as sampling points of a Riemann sum quadrature on $[0,1]$.

The condition \eqref{eq:weight_monotone} in Assumption~\ref{assm:3} basically requires that the sampling points should not be too close to each other, and the condition is used to derive the necessary a priori estimates on numerical solutions as in Theorem~\ref{thm:a_priori_estimates}.
To demonstrate the meaning of the constant $c$ and the factor $m^{-2}$ in \eqref{eq:weight_monotone}, let us illustrate with the left endpoint quadrature weights in \eqref{eq:quad_weight_left}. In this case,
\begin{align*}
	w_{k-1} - w_k = \frac{h}{\delta} \left[ w\left((k-1)\frac{h}{\delta}\right) - w\left(k\frac{h}{\delta}\right) \right] \geq \left( \min_{s\in[0,1]} -w'(s) \right) \left( \frac{h}{\delta} \right)^2 \geq c m^{-2},
\end{align*}
where the constant $c= \min_{s\in[0,1]} -w'(s) > 0$.

The condition \eqref{eq:normalization_condition} in Assumption~\ref{assm:3} is the normalization condition for the numerical quadrature weights, which is essential to the consistency between the scheme \eqref{eq:nonlocal_lwr_num}-\eqref{eq:nonlocal_lwr_num_2} and the local model \eqref{eq:lwr}.
To demonstrate potential risks when the normalization condition \eqref{eq:normalization_condition} is violated, let us consider the case $\delta=mh$ where $m$ is a fixed positive integer. Then the scheme \eqref{eq:nonlocal_lwr_num}-\eqref{eq:nonlocal_lwr_num_2} can be viewed as a $m+2$-point conservative scheme of the local model \eqref{eq:lwr} with the numerical flux function:
\begin{align*}
	g_{\mathrm{local}}(\rho_j, \cdots, \rho_{j+m}) = g\left(\rho_j, \rho_{j+1}, \sum_{k=0}^{m-1} w_k\rho_{j+k}, \sum_{k=0}^{m-1} w_k\rho_{j+k+1} \right),
\end{align*}
where $g$ is as in Assumption~\ref{assm:4}.
Suppose $\rho_j = \cdots = \rho_{j+m} = \bar{\rho}$, to make $g_{\mathrm{local}}$ consistent to the local model \eqref{eq:lwr}, it is necessary to have
\begin{align*}
	g_{\mathrm{local}}(\bar{\rho}, \cdots, \bar{\rho}) = \bar{\rho} \left(1 - \bar{\rho} \sum_{k=0}^{m-1}w_k \right) = \bar{\rho} (1 - \bar{\rho}),
\end{align*}
which requires the normalization condition \eqref{eq:normalization_condition}. In contrast, if the condition \eqref{eq:normalization_condition} is violated and $\eta \doteq \sum_{k=0}^{m-1}w_k \neq 1$, the numerical solutions will formally converge to a solution of the equation
\begin{align*}
	\partial_t\rho(t,x) + \partial_x( \rho(t,x) (1-\eta \rho(t,x)) ) = 0,
\end{align*}
other than the desired equation \eqref{eq:lwr} with $v(\rho)=1-\rho$.
This means that the absence of the normalization condition \eqref{eq:normalization_condition} for some numerical quadrature weights may lead to incorrect limit solutions when $\delta\to0$ and $h\to0$ simultaneously.
Hence, we introduce the following normalized left endpoint quadrature weights:
\begin{align}
	\mbox{[Normalized left endpoint]} \quad   w_k = \frac{w_\delta(kh)h}{\sum_{k=0}^{m-1} w_\delta(kh)h}, \quad k=0,\cdots,m-1, \label{eq:quad_weight_left_norm} 
\end{align}
and give the following proposition. 
\begin{prop}
    The normalized left endpoint quadrature weights \eqref{eq:quad_weight_left_norm} and the exact quadrature weights \eqref{eq:quad_weight_exact} both satisfy the Assumption~\ref{assm:3}, with the constant $c$ in the condition \eqref{eq:weight_monotone} given by $c= \frac{1}{1+w(0)} \min_{s\in[0,1]} -w'(s) $ and $c= \min_{s\in[0,1]} -w'(s) $, respectively.
    The left endpoint quadrature weights satisfy the condition \eqref{eq:weight_monotone} with the constant $c= \min_{s\in[0,1]} -w'(s) $ but they do not satisfy the normalization condition \eqref{eq:normalization_condition}.
\end{prop}
A comparison between the different choices of numerical quadrature weights is made through numerical experiments in Section~\ref{sec:numerical_experiments}. 

Concerning the Assumption~\ref{assm:4} on the numerical flux function $g$, with the velocity function $v(\rho)=1-\rho$, the flux function in the continuum model \eqref{eq:nonlocal_lwr} is $\rho(1-q)$, which is a quadratic polynomial of $(\rho,q)$ with the only quadratic term being $-\rho q$. It is then reasonable to assume that the numerical flux function $g$ is quadratic with its second order derivatives satisfying the condition (iii). The condition (ii) guarantees the consistency of the scheme \eqref{eq:nonlocal_lwr_num}-\eqref{eq:nonlocal_lwr_num_2} to the model \eqref{eq:nonlocal_lwr}. The condition (iv) is used to show that the scheme is monotone under all Assumptions~\ref{assm:1}-\ref{assm:5}, see the Theorem~\ref{thm:a_priori_estimates}. It is natural to ask if the results in this work can be extended to more general numerical flux functions, e.g., $g$ is not quadratic. We leave the study of such an extension to future works.

Let us mention that the numerical flux functions given in \eqref{eq:g_func_LxF}-\eqref{eq:g_func_LxF_new} all satisfy Assumption~\ref{assm:4}. For the two Lax-Friedrichs type numerical flux functions \eqref{eq:g_func_LxF} and \eqref{eq:g_func_LxF_new}, the numerical viscosity constant should satisfy $\alpha\geq2$.

We also remark that, in the case of $0<\delta\leq h$, i.e., the nonlocal horizon is within one spatial mesh cell, it holds that $q_L=\rho_L$ and $q_R=\rho_R$ by Assumption~\ref{assm:3}.
Suppose $g$ satisfies Assumption~\ref{assm:4}, the numerical flux function $g=g(\rho_L,\rho_R,\rho_L,\rho_R)$ for the local model \eqref{eq:lwr} is non-decreasing with respect to $\rho_L$ and non-increasing with respect to $\rho_R$.
Therefore, the scheme \eqref{eq:local_lwr_num} is a monotone scheme for the local model \eqref{eq:lwr}.
As a consequence, we give the following result.
\begin{prop}\label{prop:scheme_local_limit}
Under Assumptions~\ref{assm:1}, \ref{assm:3}-\ref{assm:5}, let $\rho^{\delta,h}$ be the numerical solution produced by the scheme \eqref{eq:nonlocal_lwr_num}-\eqref{eq:nonlocal_lwr_num_2} and $\rho^{0,h}$ be the one produced by \eqref{eq:local_lwr_num}. It holds that:
	\begin{align*}
		\rho^{\delta,h} = \rho^{0,h} \quad \mathrm{when} \quad 0 < \delta \leq h.
	\end{align*}
	Moreover, $\rho^{0,h}$ converges in $\mathbf{L}^1_{\mathrm{loc}}([0,+\infty)\times\mathbb{R})$ to the weak entropy solution $\rho^0$ of the local model \eqref{eq:lwr} as defined in Proposition~\ref{prop:local_limit}.
\end{prop}

\section{Proof of theorems}

This section aims to give the proofs of our main results. First, in Section~\ref{sec:maxm_principle}, we show the maximum principle for numerical solutions. Then we present an one-sided Lipschitz estimate for numerical solutions in Section~\ref{sec:tv_estimate_ac}, the monotonicity of the numerical scheme \eqref{eq:nonlocal_lwr_num}-\eqref{eq:nonlocal_lwr_num_2} and the total variation estimate for numerical solutions follow as corollaries.
These two subsections constitute the proof of Theorem~\ref{thm:a_priori_estimates}.
In Section~\ref{sec:convergence}, we let $h\to0$ and show convergence of numerical solutions to the proper nonlocal or local solution, which gives the proofs of Theorem~\ref{thm:ac} and Theorem~\ref{thm:numerical_convergence}.
In Section~\ref{sec:local_limit_num}, we give the proof of Theorem~\ref{thm:ap} on the nonlocal-to-local limit of numerical discretizations.

\subsection{Maximum principle}\label{sec:maxm_principle}

In this subsection, we aim to show the maximum principle \eqref{eq:numerical_maxm_principle} in Theorem~\ref{thm:a_priori_estimates}.
By Assumption~\ref{assm:1} and \eqref{eq:ini_data_discrete}, the numerical solution at the initial time $\{\rho_j^0\}_{j\in\mathbb{Z}}$ satisfies $0 \leq \rho_j^0 \leq 1$ for all ${j\in\mathbb{Z}}$. Then the maximum principle \eqref{eq:numerical_maxm_principle} can be proved by induction using the following Lemma~\ref{lemm:sol_range_induction}.

\begin{lemm}\label{lemm:sol_range_induction}
	Suppose all conditions in Theorem~\ref{thm:a_priori_estimates} are given, and that $0\leq\rho_{\mathrm{min}}\leq\rho_{j+k}^n\leq\rho_{\mathrm{max}}\leq1$ for $k=-1,0,1,\cdots,m$.
	Then we have
    \begin{align}\label{eq:sol_range_induction}
		\rho_{\mathrm{min}}\leq \mathcal{H}(\rho^n_{j-1},\rho^n_j,\rho^n_{j+1},\cdots,\rho^n_{j+m})\leq \rho_{\mathrm{max}},
	\end{align}
    where the operator $\mathcal{H}$ is as defined in \eqref{eq:nonlocal_lwr_num}-\eqref{eq:nonlocal_lwr_num_2}.
\end{lemm}

Let us first check the monotonicity of the scheme defined by \eqref{eq:nonlocal_lwr_num}-\eqref{eq:nonlocal_lwr_num_2}.
Denote
\begin{align*}
	\theta_j^{n,(i)} = \theta^{(i)}(q_{j-1}^n, q_j^n), \ i=1,2; \quad \theta_j^{n,(i)} = \theta^{(i)}(\rho_{j-1}^n, \rho_j^n), \ i=3,4 \quad \text{for} \quad j\in\mathbb{Z}, \ n\geq0.
\end{align*}
A direct calculation gives:
\begin{subequations}
\begin{align}
	\frac{\partial \mathcal{H}}{\partial \rho^n_{j-1}} &= \lambda \left( \theta_j^{n,(1)} + w_0 \theta_j^{n,(3)} \right); \label{eq:H_grad_1} \\
	\frac{\partial \mathcal{H}}{\partial\rho^n_j} &= 1 + \lambda \left( \theta_j^{n,(2)} - \theta_{j+1}^{n,(1)} + w_1 \theta_j^{n,(3)} + w_0 \theta_j^{n,(4)} - w_0 \theta_{j+1}^{n,(3)} \right); \label{eq:H_grad_2} \\
	\frac{\partial \mathcal{H}}{\partial \rho^n_{j+1}} &= \lambda \left(w_2 \theta_j^{n,(3)} + w_1 \theta_j^{n,(4)} - \theta_{j+1}^{n,(2)} - w_1 \theta_{j+1}^{n,(3)} - w_0 \theta_{j+1}^{n,(4)} \right); \label{eq:H_grad_3} \\
	\frac{\partial \mathcal{H}}{\partial \rho^n_{j+k}} &= \lambda \left( w_{k+1} \theta_j^{n,(3)} - w_k \theta_{j+1}^{n,(3)} + w_k \theta_j^{n,(4)} - w_{k-1} \theta_{j+1}^{n,(4)} \right), \quad k=2,\cdots,m; \label{eq:H_grad_4}
\end{align}
\end{subequations}
where we make the convention that $w_m = w_{m+1} = 0$.

In \eqref{eq:H_grad_4} that corresponds to the nonlocal dependence of the flux on the solution, it is possible that $\theta_j^{n,(3)}<0, \, \theta_j^{n,(4)}<0$ while $\theta_{j+1}^{n,(3)}=\theta_{j+1}^{n,(4)}=0$ at some point $j=j_0$, e.g., if we consider the Riemann type solution:
\begin{align*}
	\rho_j^n = 1, \ j\leq j_0; \quad \rho_j^n = 0, \ j> j_0.
\end{align*}
In this case, $\frac{\partial \mathcal{H}}{\partial \rho^n_{j+k}} < 0$ for $k=2,\cdots,m-1$.
Therefore, one can not deduce \eqref{eq:sol_range_induction} by showing \eqref{eq:nonlocal_lwr_num}-\eqref{eq:nonlocal_lwr_num_2} is a monotone scheme. Here we prove \eqref{eq:sol_range_induction} in an alternative way, which was also used in \cite{goatin2016well,friedrich2018godunov}.

\begin{proof}[Proof of Lemma~\ref{lemm:sol_range_induction}]
We observe the identity $\mathcal{H}(\rho_{\mathrm{min}},\rho_{\mathrm{min}},\rho_{\mathrm{min}},\cdots,\rho_{\mathrm{min}})=\rho_{\mathrm{min}}$ thus we can write the term $\mathcal{H}(\rho^n_{j-1},\rho^n_j,\rho^n_{j+1},\cdots,\rho^n_{j+m})-\rho_{\mathrm{min}}$ as the summation of two parts:
\begin{align*}
	\Delta \mathcal{H}_1=&\mathcal{H}(\rho^n_{j-1},\rho^n_j,\rho^n_{j+1},\rho^n_{j+2}\cdots,\rho^n_{j+m})-\mathcal{H}(\rho_{\mathrm{min}},\rho_{\mathrm{min}},\rho^n_{j+1},\rho^n_{j+2},\cdots,\rho^n_{j+m}),\\
	\Delta \mathcal{H}_2=&\mathcal{H}(\rho_{\mathrm{min}},\rho_{\mathrm{min}},\rho^n_{j+1},\rho^n_{j+2},\cdots,\rho^n_{j+m})-\mathcal{H}(\rho_{\mathrm{min}},\rho_{\mathrm{min}},\rho_{\mathrm{min}},\rho_{\mathrm{min}}\cdots,\rho_{\mathrm{min}}).
\end{align*}
By the mean value theorem,
\begin{align*}
	\Delta \mathcal{H}_1&=\sum_{k=-1,0}\frac{\partial \mathcal{H}}{\partial \rho^n_{j+k}}(\tilde{\rho}^n_{j-1},\tilde{\rho}^n_j,\rho^n_{j+1},\rho^n_{j+2}\cdots,\rho^n_{j+m})(\rho^n_{j+k}-\rho_{\mathrm{min}}),\\
	\Delta \mathcal{H}_2&=\sum_{1\leq k\leq m}\frac{\partial \mathcal{H}}{\partial \rho^n_{j+k}}(\rho_{\mathrm{min}},\rho_{\mathrm{min}},\tilde{\rho}^n_{j+1},\tilde{\rho}^n_{j+2}\cdots,\tilde{\rho}^n_{j+m})(\rho^n_{j+k}-\rho_{\mathrm{min}}),
\end{align*}
where $0\leq\rho_{\mathrm{min}}\leq\tilde{\rho}^n_{j+k}\leq\rho_{\mathrm{max}}\leq1 \ \forall k=-1,0,1,\cdots,m$.

Let us use \eqref{eq:H_grad_1}-\eqref{eq:H_grad_4} with $\theta_j^{n,(i)}$ replaced by $\tilde{\theta}_j^{n,(i)}$ that is with respect to $\tilde{\rho}^n_{j+k}$. 
By Assumption~\ref{assm:4}, we have $\tilde{\theta}_j^{n,(1)} + \tilde{\theta}_j^{n,(3)} \geq0$ giving that the term with respect to $k=-1$ in $\Delta \mathcal{H}_1$ is nonnegative. Moreover, Assumption~\ref{assm:5} implies
that the term with respect to $k=0$ in $\Delta \mathcal{H}_1$ is nonnegative.
For $\Delta \mathcal{H}_2$, we note that
\begin{align*}
	\tilde{\theta}_{j+1}^{n,(3)} = \tilde{\theta}_{j}^{n,(3)} + \gamma_{23} (\tilde{\rho}_{j+1}^n - \rho_{\mathrm{min}}) \leq \tilde{\theta}_{j}^{n,(3)}, \quad \tilde{\theta}_{j+1}^{n,(4)} = \tilde{\theta}_{j}^{n,(4)} + \gamma_{24} (\tilde{\rho}_{j+1}^n - \rho_{\mathrm{min}}) \leq \tilde{\theta}_{j}^{n,(4)},
\end{align*}
which yields that
\begin{align*}
	w_{k+1} \tilde{\theta}_j^{n,(3)} - w_k \tilde{\theta}_{j+1}^{n,(3)} + w_k \tilde{\theta}_j^{n,(4)} - w_{k-1} \tilde{\theta}_{j+1}^{n,(4)} \geq (w_{k+1} - w_k) \tilde{\theta}_j^{n,(3)} + (w_k - w_{k-1}) \tilde{\theta}_j^{n,(4)} \geq0,
\end{align*}
for $k=1,\cdots,m$.
Hence $\frac{\partial \mathcal{H}}{\partial \tilde{\rho}^n_{j+k}}\geq0$ for $k=1,\cdots,m$.
Then we deduce that 
\begin{align*}
	\mathcal{H}(\rho^n_{j-1},\rho^n_j,\rho^n_{j+1},\cdots,\rho^n_{j+m})-\rho_{\mathrm{min}}=\Delta \mathcal{H}_1+\Delta \mathcal{H}_2\geq0.
\end{align*}
Similarly one can show the upper bound estimate $\mathcal{H}(\rho^n_{j-1},\rho^n_j,\rho^n_{j+1},\cdots,\rho^n_{j+m})-\rho_{\mathrm{max}}\leq0$.
\end{proof}

\subsection{One-sided Lipschitz estimate}\label{sec:tv_estimate_ac}

We now derive an one-sided Lipschitz estimate for numerical solutions as given in Lemma~\ref{lem:lip_estimate_1}.
Then we can deduce that the scheme \eqref{eq:nonlocal_lwr_num}-\eqref{eq:nonlocal_lwr_num_2} is monotone and obtain total variation estimates for numerical solutions as given in Lemma~\ref{lem:tv_estimate_ac}.
The total variation diminishing property and the estimate \eqref{eq:numerical_tvd} in Theorem~\ref{thm:a_priori_estimates} are direct corollaries of Lemma~\ref{lem:tv_estimate_ac}.

\begin{lemm}\label{lem:lip_estimate_1}
Suppose all conditions in Theorem~\ref{thm:a_priori_estimates} are given, and that $\{\rho_j^n\}_{j\in\mathbb{Z}}^{n\geq0}$ is the numerical solution produced by the scheme \eqref{eq:nonlocal_lwr_num}-\eqref{eq:nonlocal_lwr_num_2}.
The numerical differences
\begin{align}
	r_j^n = \rho_{j+1}^n - \rho_j^n, \quad j\in\mathbb{Z}, \ n\geq0,
\end{align}
satisfy
\begin{align}
	r_j^n\geq -Lh, \quad j\in\mathbb{Z},\ n\geq0.\label{eq:num_lip_estimate_1}
\end{align}
\end{lemm}

\begin{proof}
It follows from the definition of the scheme \eqref{eq:nonlocal_lwr_num}-\eqref{eq:nonlocal_lwr_num_2} that
\begin{align}\label{eq:tmp1}
	r_j^{n+1} &= r_j^n + \lambda \left[ 2g(\rho_j^n,\rho_{j+1}^n,q_j^n,q_{j+1}^n) - g(\rho_{j-1}^n,\rho_j^n,q_{j-1}^n,q_j^n) - g(\rho_{j+1}^n,\rho_{j+2}^n,q_{j+1}^n,q_{j+2}^n) \right].
\end{align}
Noting that $g$ is a quadratic function, we can do Taylor's expansions to get
\begin{align*}
	&g(\rho_j^n,\rho_{j+1}^n,q_j^n,q_{j+1}^n) - g(\rho_{j-1}^n,\rho_j^n,q_{j-1}^n,q_j^n) \\ 
	&\;\; = \theta_j^{n,(1)} r_{j-1}^n + \theta_j^{n,(2)} r_j^n + \theta_j^{n,(3)} (q_j^n-q_{j-1}^n) + \theta_j^{n,(4)} (q_{j+1}^n-q_j^n) \\
	&\quad + \gamma_{13}r_{j-1}^n(q_j^n-q_{j-1}^n) + \gamma_{14}r_{j-1}^n(q_{j+1}^n-q_j^n) + \gamma_{23}r_j^n(q_j^n-q_{j-1}^n) + \gamma_{24}r_j^n(q_{j+1}^n-q_j^n), \\
    &g(\rho_{j+1}^n,\rho_{j+2}^n,q_{j+1}^n,q_{j+2}^n) - g(\rho_{j-1}^n,\rho_j^n,q_{j-1}^n,q_j^n) \\
    &\;\;	= \theta_j^{n,(1)} (r_{j-1}^n+r_j^n) + \theta_j^{n,(2)} (r_j^n+r_{j+1}^n) + \theta_j^{n,(3)} (q_{j+1}^n-q_j^n+q_j^n-q_{j-1}^n)\\
    &\quad + \theta_j^{n,(4)} (q_{j+2}^n-q_{j+1}^n+q_{j+1}^n-q_j^n) + \gamma_{13} (r_{j-1}^n+r_j^n)(q_{j+1}^n-q_j^n+q_j^n-q_{j-1}^n)\\
    &\quad + \gamma_{14} (r_{j-1}^n+r_j^n) (q_{j+2}^n-q_{j+1}^n+q_{j+1}^n-q_j^n) 
    	+ \gamma_{23} (r_j^n+r_{j+1}^n) (q_{j+1}^n-q_j^n+q_j^n-q_{j-1}^n) \\
     &\quad + \gamma_{24} (r_j^n+r_{j+1}^n) (q_{j+2}^n-q_{j+1}^n+q_{j+1}^n-q_j^n),
\end{align*}
where
$
	\{\theta_j^{n,(i)} = \theta^{(i)}(q_{j-1}^n, q_j^n)\}_{i=1}^2$ and
 $\{\theta_j^{n,(i)} = \theta^{(i)}(\rho_{j-1}^n, \rho_j^n)\}_{i=3}^4$ for
 $j\in\mathbb{Z}$ and $ n\geq0$.
Moreover, from the definition of $q_j^n$ give in \eqref{eq:discrete_nonlocal_density} we obtain:
\begin{align*}
	q_{j+1}^n - q_j^n = \sum_{k=0}^{m-1} w_k r_{j+k}^n.
\end{align*}
Therefore \eqref{eq:tmp1} can be rewritten as
\begin{align*}
	& r_j^{n+1} = \lambda \theta_j^{n,(1)} r_{j-1}^n + \left( 1 + \lambda \theta_j^{n,(2)} - \lambda \theta_j^{n,(1)} \right) r_j^n - \lambda \theta_j^{n,(2)} r_{j+1}^n \\
	&\quad + \lambda \left( \theta_j^{n,(3)} + \gamma_{13}r_{j-1}^n + (\gamma_{23}-\gamma_{13})r_j^n -\gamma_{23}r_{j+1}^n \right) \sum_{k=0}^{m-1} w_k r_{j+k-1}^n \\
	&\quad + \lambda \left( \theta_j^{n,(4)}-\theta_j^{n,(3)} + (\gamma_{14}-\gamma_{13}) r_{j-1}^n + (\gamma_{24}-\gamma_{13}-\gamma_{14}-\gamma_{23}) r_j^n - (\gamma_{23}+\gamma_{24}) r_{j+1}^n \right) \sum_{k=0}^{m-1} w_k r_{j+k}^n \\
	&\quad - \lambda \left( \theta_j^{n,(4)} + \gamma_{14}r_{j-1}^n + (\gamma_{14}+\gamma_{24})r_j^n + \gamma_{24}r_{j+1}^n \right) \sum_{k=0}^{m-1} w_k r_{j+k+1}^n.
\end{align*}
In the above expression, $r_j^{n+1}$ is represented as a linear combination of $r_{j-1}^n,\cdots,r_{j+m}^n$.
By a direct calculation, the summation of the coefficients before the terms $r_{j-1}^n,\cdots,r_{j+m}^n$ is
\begin{align*}
	S = 1 - 2\lambda \left( (\gamma_{13}+\gamma_{14})r_j^n + (\gamma_{23}+\gamma_{24})r_{j+1}^n \right),
\end{align*}
where the fact $\gamma_{13}+\gamma_{14}+\gamma_{23}+\gamma_{24}=-1$ is used.

Since the summation does not equal one, we split two quadratic terms with respect to $r_j^n$ and $r_{j+1}^n$, which gives the form
\begin{align}\label{eq:quad_comb}
	r_j^{n+1} = \sum_{-1\leq k\leq m}c_{j,k}^n r_{j+k}^n - 2\lambda (\gamma_{13}+\gamma_{14}) (r_j^n)^2 - 2\lambda (\gamma_{23}+\gamma_{24}) (r_{j+1}^n)^2, 
\end{align}
such that $\sum_{-1\leq k\leq m}c_{j,k}^n=1$.
The coefficients $\{c_{j,k}^n\}_{-1\leq k\leq m}$ are given by:
\begin{subequations}\label{eq:coefficients}
\begin{align}
    c_{j,-1}^n &= \lambda \theta_j^{n,(1)} + \lambda w_0 \left( \theta_j^{n,(3)} + \gamma_{13}r_{j-1}^n + (\gamma_{23}-\gamma_{13})r_j^n -\gamma_{23}r_{j+1}^n \right); \\
    c_{j,0}^n &= 1 + \lambda \left( \theta_j^{n,(2)} - \theta_j^{n,(1)} \right) + \lambda p_{j,0}^n + 2\lambda (\gamma_{13}+\gamma_{14}) r_j^n; \\
    c_{j,1}^n &= -\lambda \theta_j^{n,(2)} + \lambda p_{j,1}^n + 2\lambda (\gamma_{23}+\gamma_{24}) r_{j+1}^n; \\
    c_{j,k}^n &= \lambda p_{j,k}^n, \quad k=2,\cdots,m;
\end{align}
\end{subequations}
where
\begin{align}\label{eq:coefficients_p}
	p_{j,k}^n =& w_{k+1} \left( \theta_j^{n,(3)} + \gamma_{13}r_{j-1}^n + (\gamma_{23}-\gamma_{13})r_j^n -\gamma_{23}r_{j+1}^n \right) \\
	& + w_k \left( \theta_j^{n,(4)}-\theta_j^{n,(3)} + (\gamma_{14}-\gamma_{13}) r_{j-1}^n + (\gamma_{24}-\gamma_{13}-\gamma_{14}-\gamma_{23}) r_j^n - (\gamma_{23}+\gamma_{24}) r_{j+1}^n \right) \notag\\
	& - w_{k-1} \left( \theta_j^{n,(4)} + \gamma_{14}r_{j-1}^n + (\gamma_{14}+\gamma_{24})r_j^n + \gamma_{24}r_{j+1}^n \right), \notag
\end{align}
and we make the convention that $w_{-1} = w_m = w_{m+1} = 0$.

The initial one-sided Lipschitz condition \eqref{eq:ini_lip_const} gives $r_j^0\geq -Lh$ for all $j\in\mathbb{Z}$.
We next show that if \eqref{eq:num_lip_estimate_1} holds for any $n\geq0$,  then it is also true for $n+1$. Then \eqref{eq:num_lip_estimate_1} follows by induction.

Let us use \eqref{eq:quad_comb}-\eqref{eq:coefficients_p}. 
By Assumptions~\ref{assm:3}-\ref{assm:5},
we have $c_{j,k}^n\geq0$ for $k=-1,0$ and $-\lambda \theta_j^{n,(2)} + 2\lambda (\gamma_{23}+\gamma_{24}) r_{j+1}^n\geq0$. To show $c_{j,k}^n\geq0$ for all $-1\leq k\leq m$, it suffices to show $p_{j,k}^n\geq0$ for all $k=1,\cdots,m$.
By Assumptions~\ref{assm:3}-\ref{assm:5}, we have that
\begin{align*}
	w_{k+1} \leq w_k \leq w_{k-1} \leq w(0)m^{-1}, \quad 
	w_{k-1}-w_k \geq cm^{-2}, \quad w_k-w_{k+1} \geq cm^{-2},
\end{align*}
and $\theta_j^{n,(3)} + \theta_j^{n,(4)} \leq -\rho_{\mathrm{min}}$, where the constant $c$ is as in \eqref{eq:weight_monotone} and the constant $\rho_{\mathrm{min}}$ is as in \eqref{eq:ini_lip_const}. Then we deduce that
\begin{align*}
	p_{j,k}^n \geq& -2 w(0)m^{-1} [(\gamma_{13}+\gamma_{14})r_j^n + (\gamma_{23}+\gamma_{24})r_{j+1}^n] +  cm^{-2} (\theta_j^{n,(3)} + \theta_j^{n,(4)}) \\
	\geq& -2 w(0)m^{-1}Lh +  cm^{-2} \rho_{\mathrm{min}} 
	\geq m^{-2} (c\rho_{\mathrm{min}} - 2w(0)\delta L) 
	\geq 0,
\end{align*}
provided $0<\delta\leq\delta_0=\frac{c\rho_{\mathrm{min}}}{2Lw(0)}$.

Now we have that the coefficients $\{c_{j,k}^n\}_{-1\leq k\leq m}$ are all nonnegative and the sum of the coefficients $\sum_{-1\leq k\leq m}c_{j,k}^n=1$.
Therefore $r_j^{n+1}$ is a convex combination of $r_{j-1}^n,r_j^n,r_{j+1}^n,\cdots,r_{j+m}^n$ plus the nonnegative quadratic terms $- 2\lambda (\gamma_{13}+\gamma_{14}) (r_j^n)^2 - 2\lambda (\gamma_{23}+\gamma_{24}) (r_{j+1}^n)^2$. Hence we have:
\begin{align*}
	\inf_{j\in\mathbb{Z}}r_j^{n+1}\geq\inf_{j\in\mathbb{Z}}r_j^n\geq -Lh,
\end{align*}
which completes the proof.
\end{proof}

Based on Lemma~\ref{lem:lip_estimate_1}, a more careful analysis gives the following sharper estimate corresponding to the entropy condition \eqref{eq:oleinik}.

\begin{lemm}\label{lem:lip_estimate_2}
	Suppose all conditions in Theorem~\ref{thm:a_priori_estimates} are given, and that $0<h<h_0$ with $h_0>0$ only depending on $1-\lambda\sum_{i=1}^4\norm{\theta^{(i)}}_\infty$ and $\frac{c\rho_{\mathrm{min}}}{2Lw(0)}-\delta$. We have:
	\begin{align}
	 	L^n\leq\frac{1}{\frac{1}{L^0}+2n\tau}\leq\frac{1}{2n\tau},\quad n\geq1,\label{eq:num_lip_estimate_2}
	\end{align}
	where
	\begin{align}
		L^n\triangleq\sup_{j\in\mathbb{Z}}\max\left\{-\frac{r_j^n}{h},0\right\},\quad n\geq0.
	\end{align}
\end{lemm}

\begin{proof}
We still start with \eqref{eq:quad_comb}. For $k\neq0,1$, we use the estimate
\begin{align}\label{eq:c_tmp1}
	c_{j,k}^nr_{j+k}^n\geq -c_{j,k}^nL^nh.
\end{align}
For $k=0$ and $k=1$, we consider the following quadratic functions:
\begin{align*}
	b_0(r_j^n) \doteq c_{j,0}^nr_j^n - 2\lambda (\gamma_{13}+\gamma_{14}) (r_j^n)^2 , \quad b_1(r_{j+1}^n) \doteq c_{j,1}^nr_{j+1}^n - 2\lambda (\gamma_{23}+\gamma_{24}) (r_{j+1}^n)^2,
\end{align*}
respectively.
One can verify that
\begin{align*}
	b'_0(r_j^n) &= c_{j,0}^n - 4\lambda (\gamma_{13}+\gamma_{14}) r_j^n \geq c_{j,0}^n + 4\lambda (\gamma_{13}+\gamma_{14})L^nh \geq C_0 - 4\lambda Lh, \\
	b'_1(r_{j+1}^n) &= c_{j,1}^n - 4\lambda (\gamma_{23}+\gamma_{24}) r_{j+1}^n \geq c_{j,1}^n - 4\lambda (\gamma_{23}+\gamma_{24})L^nh \geq C_1 - 4 \lambda Lh,
\end{align*}
when $r_j^n, r_{j+1}^n \geq -L^nh$, where the constant $C_0>0$ only depends on $1-\lambda\sum_{i=1}^4\norm{\theta^{(i)}}_\infty$ and the constant $C_1>0$ only depends on $\frac{c\rho_{\mathrm{min}}}{2Lw(0)}-\delta$. Therefore there exists $h_0>0$ only depending on $1-\lambda\sum_{i=1}^4\norm{\theta^{(i)}}_\infty$ and $\frac{c\rho_{\mathrm{min}}}{2Lw(0)}-\delta$ such that $b'_0(r_j^n)\geq0, b'_1(r_{j+1}^n)\geq0$ whenever $h<h_0$.
In this case, we have
\begin{align}\label{eq:c_tmp2}
	b_0(r_j^n) \geq -c_{j,0}^nL^nh - 2\lambda (\gamma_{13}+\gamma_{14}) (L^nh)^2 , \quad b_1(r_{j+1}^n) \geq -c_{j,1}^nL^nh - 2\lambda (\gamma_{23}+\gamma_{24}) (L^nh)^2.
\end{align}
Summing up \eqref{eq:c_tmp1} for $k\neq0,1$ and \eqref{eq:c_tmp2} for $k=0,1$, and noting that $\gamma_{13}+\gamma_{14}+\gamma_{23}+\gamma_{24}=-1$, we obtain:
	\begin{align*}
		r_j^{n+1}\geq-\left(\sum_{-1\leq k\leq m}c_{j,k}^n\right)L^nh+2\lambda(L^nh)^2=-L^nh+2(L^n)^2h\tau,
	\end{align*}
	which yields
$		L^{n+1}\leq L^n-2(L^n)^2\tau$.
	Then \eqref{eq:num_lip_estimate_2} follows by induction.
\end{proof}

Now let us go back to check the monotonicity of the scheme \eqref{eq:nonlocal_lwr_num}-\eqref{eq:nonlocal_lwr_num_2}. With the derived one-sided Lipschitz estimate \eqref{eq:num_lip_estimate_1}, a calculation similar to that in the proof of Lemma~\ref{lem:lip_estimate_1} gives:
\begin{align*}
	\frac{\partial \mathcal{H}}{\partial \rho^n_{j+k}} &= \lambda \left( w_{k+1} \theta_j^{n,(3)} - w_k \theta_{j+1}^{n,(3)} + w_k \theta_j^{n,(4)} - w_{k-1} \theta_{j+1}^{n,(4)} \right) \\
	&\geq \lambda m^{-2} (c\rho_{\mathrm{min}} - 2w(0)\delta L) \geq0,
\end{align*}
for $k=2,\cdots,m$. In this case, the scheme \eqref{eq:nonlocal_lwr_num}-\eqref{eq:nonlocal_lwr_num_2} is monotone with respect to each of its arguments. As a direct corollary, it is total variation diminishing (TVD). So we have the following lemma.

\begin{lemm}\label{lem:tv_estimate_ac}
	Under the same conditions as in Lemma~\ref{lem:lip_estimate_1}, the numerical solution $\{\rho_j^n\}_{j\in\mathbb{Z}}^{n\geq0}$ produced by the scheme \eqref{eq:nonlocal_lwr_num}-\eqref{eq:nonlocal_lwr_num_2} satisfies:
	\begin{align}
		\sum_{j\in\mathbb{Z}}|r_j^n|\leq& \sum_{j\in\mathbb{Z}}|r_j^0|\leq \mathrm{TV}(\rho_0),\quad n\geq0;\label{eq:tv_estimate_ac}\\
		\sum_{j\in\mathbb{Z}}|\rho_j^{n+1}-\rho_j^n|\leq& \lambda \norm{\nabla g}_\infty \sum_{j\in\mathbb{Z}}|r_j^n|\leq  \mathrm{TV}(\rho_0),\quad n\geq0.\label{eq:tv_estimate_2_ac}
	\end{align}
\end{lemm}

The proof of the Lemma is similar to that given in \cite{leveque2002finite} for monotone schemes of scalar conservation laws. The total variation estimate \eqref{eq:numerical_tvd} follows immediately from the above lemma.

\subsection{Convergence}\label{sec:convergence}

In this subsection, we are going to give the proofs of Theorem~\ref{thm:ac} and Theorem~\ref{thm:numerical_convergence}.
We recall that the numerical solution is defined as:
\begin{align}\label{eq:sol_piece_const}
	\rho^{\delta,h}(t,x)=\sum_{j\in\mathbb{Z}}\sum_{n=0}^\infty\rho_j^n\mathbf{1}_{\mathcal{C}_j\times\mathcal{T}^n}(t,x),
\end{align}
where $\mathcal{C}_j=(x_{j-1/2},x_{j+1/2}) \ \forall j\in\mathbb{Z}$ and $\mathcal{T}^n=(t^n,t^{n+1})\ \forall n\geq0$.

\begin{proof}[Proof of Theorem~\ref{thm:ac}]
Let us consider the family of numerical solutions $
\{ \rho^{\delta,h}\}_{0 < \delta \leq \delta_0, 0 < h < 1}$,
where $\delta_0$ is as in \eqref{eq:delta_condition}.
Theorem~\ref{thm:a_priori_estimates} gives the a priori $\mathbf{L}^\infty$ and total variation estimates on $\rho^{\delta,h}$, thus the family of numerical solutions is uniformly bounded in $\mathbf{BV}_{\mathrm{loc}}([0,+\infty) \times \mathbb{R})$. Thus, it is precompact in the $\mathbf{L}^1_{\mathrm{loc}}$ norm (see \cite{evans2018measure}), and there exists a sequence           
 $\{\rho^{\delta_l, h_l}\}$ converging in $\mathbf{L}^1_{\mathrm{loc}}([0,+\infty)\times\mathbb{R})$ to a limit function $\rho^{*}$ as $\delta_l\to0, h_l\to0$ simultaneously. Noting the uniqueness of the entropy solution, to show the convergence of $\rho^{\delta,h}$ when $\delta\to0$ and $h\to0$ along an arbitrary path, we only need to show $\rho^{*}$ satisfies both the weak form \eqref{eq:local_sol} and the entropy condition \eqref{eq:oleinik}.

For any test function $\phi\in\mathbf{C}^1_{\mathrm{c}}\left([0,+\infty)\times\mathbb{R}\right)$, we denote $\phi_j^n=\phi(t^n,x_j)$ for all $j\in\mathbb{Z}$ and $n\geq0$. Multiplying the scheme \eqref{eq:nonlocal_lwr_num}-\eqref{eq:nonlocal_lwr_num_2} by $\phi_j^nh$, summing over all $j\in\mathbb{Z}$ and $n\geq0$, and applying summation by parts, we obtain:
\begin{align}
	h\tau \sum_{n\geq1} \sum_{j\in\mathbb{Z}} \frac{\phi_j^n-\phi_j^{n-1}}{\tau} \rho_j^n + h\tau \sum_{n\geq0} \sum_{j\in\mathbb{Z}} \frac{\phi_{j+1}^n-\phi_j^n}{h} g(\rho_j^n,\rho_{j+1}^n,q_j^n,q_{j+1}^n) + h \sum_{j\in\mathbb{Z}} \phi_j^0\rho_j^0 = 0. \label{eq:dis_weak_form}
\end{align}

When $h\to0$, given the assumptions on $\phi$,  it is straightforward to show that:
\begin{align}
h\sum_{j\in\mathbb{Z}}\phi_j^0\rho_j^0\to&\int_{\mathbb{R}}\rho_0(x)\phi(0,x)\,dx,\label{eq:conv_1}\\
h\tau\sum_{n\geq1}\sum_{j\in\mathbb{Z}}\frac{\phi_j^n-\phi_j^{n-1}}{\tau}\rho_j^n\to&\int_0^\infty\int_{\mathbb{R}}\rho^*(t,x)\partial_t\phi(t,x)\,dxdt.\label{eq:conv_2}
\end{align}
We need to show:
\begin{align*}
	h\tau \sum_{n\geq0} \sum_{j\in\mathbb{Z}} \frac{\phi_{j+1}^n-\phi_j^n}{h} g(\rho_j^n,\rho_{j+1}^n,q_j^n,q_{j+1}^n) \to \int_0^\infty\int_{\mathbb{R}} \rho^\ast(t,x)v\left(\rho^\ast(t,x)\right)\partial_x\phi(t,x) \,dxdt.
\end{align*}
Using the total variation estimate \eqref{eq:tv_estimate_ac}, we have
\begin{align}
&\sum_{j\in\mathbb{Z}} |g(\rho_j^n,\rho_{j+1}^n,q_j^n,q_{j+1}^n) - \rho_j v(q_j^n)| 
	= \sum_{j\in\mathbb{Z}} |g(\rho_j^n,\rho_{j+1}^n,q_j^n,q_{j+1}^n) - g(\rho_j^n,\rho_j^n,q_j^n,q_j^n)| \notag\\
& \qquad	\leq \norm{\nabla g}_\infty \sum_{j\in\mathbb{Z}} |\rho_{j+1}^n-\rho_j^n|+|q_{j+1}^n-q_j^n|
	\leq 2\norm{\nabla g}_\infty \mathrm{TV}(\rho_0), \label{eq:conv3}
\end{align}
and
\begin{align*}
& \sum_{j\in\mathbb{Z}} |\rho_j v(q_j^n) - \rho_j v(\rho_j^n)|
	\leq \sum_{j\in\mathbb{Z}} |q_j^n-\rho_j^n| 
	\leq \sum_{k=0}^{m-1} w_k \sum_{j\in\mathbb{Z}} |\rho_{j+k}^n-\rho_j^n| \\
&\qquad	\leq \left( \sum_{k=1}^{m-1} kw_k \right) \sum_{j\in\mathbb{Z}} |\rho_{j+1}^n-\rho_j^n| 
	\leq \frac{mw(0)}{2} \mathrm{TV}(\rho_0).
\end{align*}
Therefore we have that the difference between the summations 
\begin{align*}
& \left|
h\tau \sum_{n\geq0} \sum_{j\in\mathbb{Z}} \frac{\phi_{j+1}^n-\phi_j^n}{h} g(\rho_j^n,\rho_{j+1}^n,q_j^n,q_{j+1}^n)
- 
h\tau \sum_{n\geq0} \sum_{j\in\mathbb{Z}} \frac{\phi_{j+1}^n-\phi_j^n}{h} \rho_j^n v(\rho_j^n) \right|\\
&\quad \leq
	C(\phi) \left( 2\norm{\nabla g}_\infty \mathrm{TV}(\rho_0) h + \frac{w(0)}{2} \mathrm{TV}(\rho_0) \delta \right) \to 0, 
\end{align*}
as $\delta\to0,h\to0$, where $C(\phi)>0$ is a constant only depending on $\phi$.
Then we can pass the limit
\begin{align*}
h\tau \sum_{n\geq0} \sum_{j\in\mathbb{Z}} \frac{\phi_{j+1}^n-\phi_j^n}{h} \rho_j^n v(\rho_j^n)    
\to 
\int_0^\infty\int_{\mathbb{R}} \rho^\ast(t,x)v\left(\rho^\ast(t,x)\right)\partial_x\phi(t,x) \,dxdt
\end{align*}
as $h\to0$ by using the $\mathbf{L}^1_{\mathrm{loc}}$ convergence from $\rho^{\delta,h}$ to $\rho^*$, and the a priori $\mathbf{L}^\infty$ bound as given in \eqref{eq:numerical_maxm_principle}, and deduce that $\rho^*$ satisfies \eqref{eq:local_sol}.
	 
For the entropy condition, let us consider numerical solutions $\tilde{\rho}^{\delta,h}$ that are constructed by linear interpolation rather than the piecewise constant reconstruction as defined in \eqref{eq:sol_piece_const}. Then by Lemma~\ref{lem:lip_estimate_2}, $\tilde{\rho}^{\delta,h}$ satisfies the one-sided Lipschitz estimate:
\begin{align}\label{eq:dis_olenik}
	-\frac{\tilde{\rho}^{\delta,h}(t,y)-\tilde{\rho}^{\delta,h}(t,x)}{y-x}\leq \frac{1}{2t} \quad \forall x\neq y\in\mathbb{R},\,t>0.
\end{align}
Noting that  $\tilde{\rho}^{\delta,h}$ converges to the same limit function $\rho^{*}$ pointwise, we can show that $\rho^{*}$ satisfies the Oleinik's entropy condition \eqref{eq:oleinik} by passing the limit on \eqref{eq:dis_olenik}.
\end{proof}

To prove Theorem~\ref{thm:numerical_convergence}, we first prove the following lemma.
\begin{lemm}\label{lemm:numerical_convergence}
Under Assumptions~\ref{assm:1}-\ref{assm:5},
and that $\delta$ satisfies the condition \eqref{eq:delta_condition}.
When $\delta\to\delta_*>0$ and $h\to0$,
the numerical solution $\rho^{\delta,h}$ produced by the scheme \eqref{eq:nonlocal_lwr_num}-\eqref{eq:nonlocal_lwr_num_2} converges in $\mathbf{L}^1_{\mathrm{loc}}([0,+\infty)\times\mathbb{R})$ to the weak solution $\rho^{\delta_*}$ of the nonlocal LWR model \eqref{eq:nonlocal_lwr} as defined in Proposition~\ref{prop:nonlocal_sol}.
\end{lemm}
\begin{proof}
Similarly as in the proof of Theorem~\ref{thm:ac}, when taking the limit $\delta\to\delta_*$ and $h\to0$, there exists a sequence $\{\rho^{\delta_l,h_l}\}$ converging to a limit function $\rho^{**}$ in the $\mathbf{L}^1_{\mathrm{loc}}$ norm with $\delta_l\to\delta_*, h_l\to0$. Noting that Proposition~\ref{prop:nonlocal_sol} already gives the solution uniqueness, we only need to show that the limit function $\rho^{**}$ satisfies the weak form \eqref{eq:nonlocal_sol}.

With similar calculations to those in the proof of Theorem~\ref{thm:ac}, we have \eqref{eq:dis_weak_form}-\eqref{eq:conv_2} for $\rho^{**}$. But here we only use \eqref{eq:conv3} and the convergence:
\begin{align}\label{eq:conv4}
	\sum_{j\in\mathbb{Z}}\sum_{n=0}^\infty q_j^n\mathbf{1}_{\mathcal{C}_j\times\mathcal{T}^n}(t,x) \to \int_0^\delta\rho^{\delta,h}(t,x+s)w_\delta(s)\,ds,
\end{align}
in the $\mathbf{L}^1_{\mathrm{loc}}$ norm. The proof of \eqref{eq:conv4} is similar to that given in \cite{Blandin2016}, we omit the details here.
Then we have
\begin{align*}
& h\tau\sum_{n\geq0}\sum_{j\in\mathbb{Z}}\frac{\phi_{j+1}^n-\phi_j^n}{h}g(\rho_j^n,\rho_{j+1}^n,q_j^n,q_{j+1}^n)\\
&\qquad \to\int_0^\infty\int_{\mathbb{R}}\rho^{**}(t,x) v\left(\int_0^\delta\rho^{**}(t,x+s)w_\delta(s)\,ds\right)\partial_x\phi(t,x)\,dxdt,
\end{align*}
which implies that $\rho^{**}$ satisfies \eqref{eq:nonlocal_sol}.
\end{proof}

We now give the proof of Theorem~\ref{thm:numerical_convergence}.

\begin{proof}[Proof of Theorem~\ref{thm:numerical_convergence}]

For any bounded set $U\subset[0,+\infty)\times\mathbb{R}$, suppose \eqref{eq:uniform_numerical_convergence} is not true, there exists a sequence of $\delta_l$ and $h_l$ where $\delta_l\in(0,\delta_0]$ and $h_l\to0$ as $l\to\infty$, and $\varepsilon>0$, such that
	\begin{align*}
		\norm{\rho^{\delta_l,h_l} - \rho^{\delta_l}}_{\mathbf{L}^1(U)} \geq \varepsilon.
	\end{align*}
	By possibly selecting a subsequence we suppose $\delta_l\to \delta_*\in[0,\delta_0]$.
	If $\delta_l\to0$, both $\rho^{\delta_l,h_l}$ and $\rho^{\delta_l}$ converge to $\rho^0$; If $ \delta_l\to \delta_*>0$, by Lemma~\ref{lemm:numerical_convergence}, $\rho^{\delta_l,h_l}\to\rho^{\delta_*}$, and by applying the same arguments on continuum solutions, it holds that $\rho^{\delta_l}\to\rho^{\delta_*}$. In either case there is a contradiction.
	
\end{proof}

\subsection{Local limit of numerical discretizations}\label{sec:local_limit_num}

We now present the proof of Theorem~\ref{thm:ap}.

\begin{proof}[Proof of Theorem~\ref{thm:ap}]
	For any bounded set $U\subset[0,+\infty)\times\mathbb{R}$, suppose \eqref{eq:ap_conv} is not true, there exists a sequence of $\delta_l$ and $h_l$ where $h_l\in(0,h_0]$ and $\delta_l\to0$ as $l\to\infty$, and $\varepsilon>0$, such that
	\begin{align*}
		\norm{\rho^{\delta_l,h_l} - \rho^{0,h_l}}_{\mathbf{L}^1(U)} \geq \varepsilon.
	\end{align*}
	By possibly selecting a subsequence we suppose $h_l\to h_*\in[0,h_0]$.
	If $h_l\to0$, both $\rho^{\delta_l,h_l}$ and $\rho^{0,h_l}$ converge to $\rho^0$; If $ h_l\to h_*>0$, by Proposition~\ref{prop:scheme_local_limit} it holds that $\rho^{\delta_l,h_l}=\rho^{0,h_l}$ when $l$ is large enough. In either case there is a contradiction.
\end{proof}

\section{Numerical experiments}\label{sec:numerical_experiments}

In this section, we test the presented numerical scheme \eqref{eq:nonlocal_lwr_num}-\eqref{eq:nonlocal_lwr_num_2} in several numerical experiments to demonstrate the established results.
In the implementation of the scheme \eqref{eq:nonlocal_lwr_num}-\eqref{eq:nonlocal_lwr_num_2},
the numerical flux function $g$ is chosen from the ones given in \eqref{eq:g_func_LxF}-\eqref{eq:g_func_LxF_new},
and the numerical quadrature weights $\{w_k\}_{0\leq k\leq m-1}$ are chosen from the ones given in \eqref{eq:quad_weight_left}-\eqref{eq:quad_weight_left_norm}-\eqref{eq:quad_weight_exact}.
We fix the CFL ratio $\lambda=0.25$. For the Lax-Friedrichs type numerical flux functions \eqref{eq:g_func_LxF} and \eqref{eq:g_func_LxF_new}, we fix the numerical viscosity constant $\alpha=2$.
In all but the final experiments, we use the linear decreasing kernel $w_\delta(s) = \frac{2}{\delta^2}(\delta - s)$.
Assuming $\delta=mh$ where $m$ is a positive integer, the numerical quadrature weights for the linear decreasing kernel computed from \eqref{eq:quad_weight_left}-\eqref{eq:quad_weight_left_norm}-\eqref{eq:quad_weight_exact} are given respectively by
\begin{itemize}
	\item (Left endpoint) $w_k = \frac{2(m-k)}{m^2}$ for $0 \leq k \leq m-1$, with $\sum_{k=0}^{m-1} w_k = 1 + \frac1m$;
     \item (Normalized left endpoint) $w_k = \frac{2(m-k)}{m(m+1)}$ for $0 \leq k \leq m-1$, with $\sum_{k=0}^{m-1} w_k = 1$;
     \item (Exact quadrature) $w_k = \frac{2(m-k)-1}{m^2}$ for $0 \leq k \leq m-1$, with $\sum_{k=0}^{m-1} w_k = 1$.
\end{itemize}

The velocity function is chosen to be $v(\rho)=1-\rho$. Two sets of initial data $\rho_0$ are used, one is a bell-shaped curve:
\begin{align}\label{eq:ini_bellshape}
	\rho_0(x) = 0.4 + 0.4 \exp\left(-100(x-0.5)^2\right), \quad x\in\mathbb{R},
\end{align}
while the other represents the Riemann data:
\begin{align}\label{eq:ini_riemann}
	\rho_0(x) = \begin{dcases}
		\rho_L, \quad x < 0.5 \\
		\rho_R, \quad x > 0.5
	\end{dcases} , \quad x\in\mathbb{R},
\end{align}
we take $\rho_L=0.1$ and $\rho_R=0.6$ in all the experiments.
The numerical solutions are presented on the spatial domain $x\in[0,1]$ and in the time horizon $t\in[0,1]$ even though
the numerical computations are done on a larger spatial domain with the constant extension on both sides.
In the first three experiments, we examine the asymptotically compatibility and uniform numerical convergence of the scheme \eqref{eq:nonlocal_lwr_num}-\eqref{eq:nonlocal_lwr_num_2} with different numerical quadrature weights.
In the last experiment, we test the scheme with different choices of the nonlocal kernel.

{\bf Experiment 1.} 
We first present numerical solutions $\rho^{\delta,h}$ computed with the Lax-Friedrichs numerical flux function \eqref{eq:g_func_LxF} and different numerical quadrature weights.
For each initial data and each set of numerical quadrature weights, we compute the numerical solution $\rho^{\delta,h}$ with $\delta=0.005, h=0.001$ and plot its snapshots at selected times $t=0,0.5,1$.
Moreover, the snapshot of the numerical solution $\rho^{\delta,h}$ at time $t=1$ is compared with that of the solution $\rho^0$ of the local model \eqref{eq:lwr}.
In this experiment, the local solution $\rho^0$ is also computed numerically because the analytical solution is not always available. The numerical computation is done on a fine grid with $h=0.0002$ using a Lax-Friedrichs  scheme for \eqref{eq:lwr} with the numerical flux function
\begin{align}\label{eq:local_LxF}
	g_{\mathrm{local}}(\rho_L,\rho_R) = \frac12 (\rho_L v(\rho_L) + \rho_R v(\rho_R)) + \frac\alpha2 (\rho_L - \rho_R),
\end{align}
that is the local counterpart of \eqref{eq:g_func_LxF}.
The snapshot of the local solution $\rho^0$ at time $t=1$ is plotted with dashed line.
See Figure~\ref{fig:exp_1}.

\begin{figure}[htbp!]
\centering
	\begin{subfigure}{.32\textwidth}
	\includegraphics[width=\textwidth]{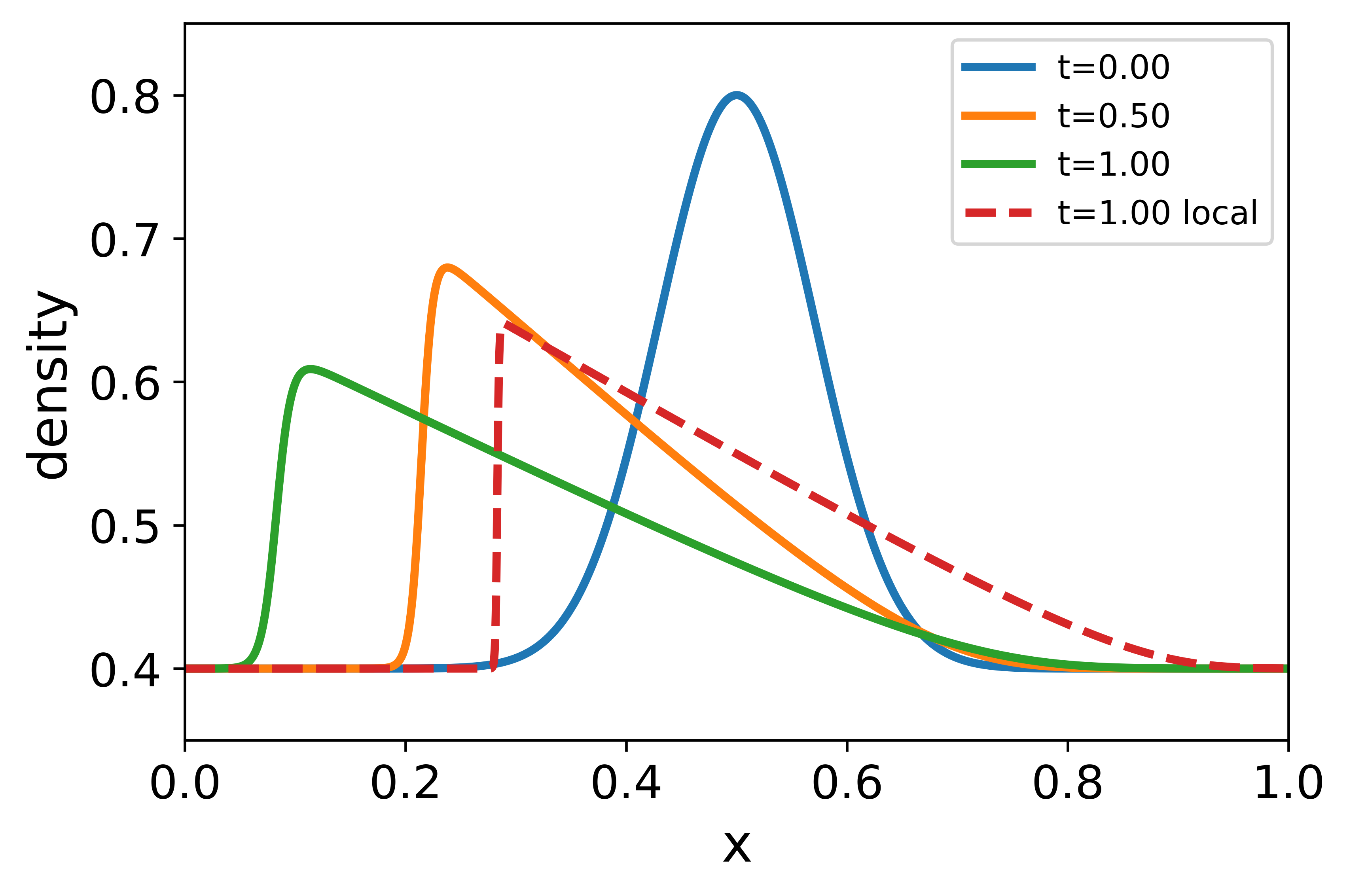}
	\end{subfigure}
	\begin{subfigure}{.32\textwidth}
	\includegraphics[width=\textwidth]{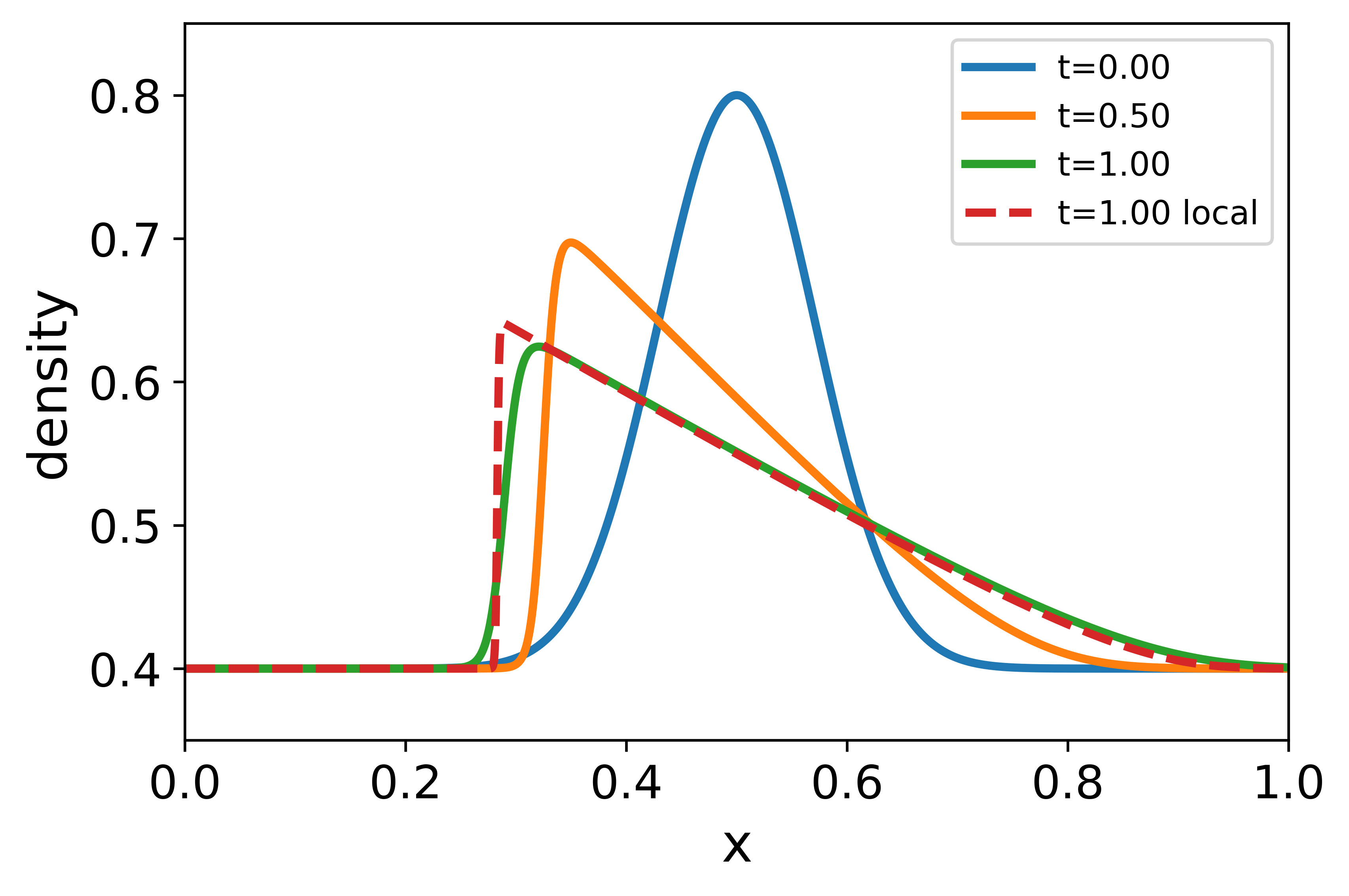}
	\end{subfigure}
	\begin{subfigure}{.32\textwidth}
	\includegraphics[width=\textwidth]{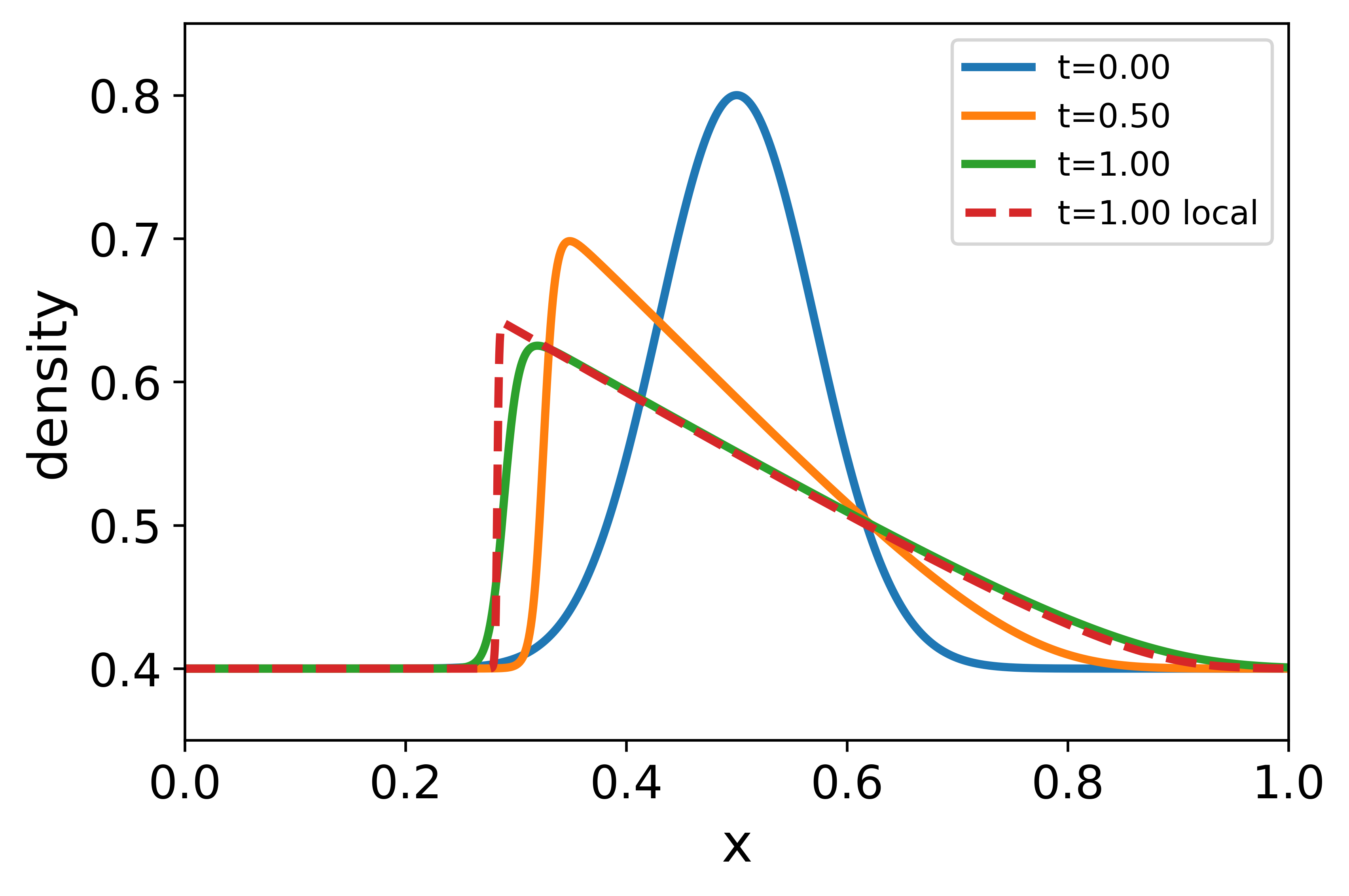}
	\end{subfigure}
	\begin{subfigure}{.32\textwidth}
	\includegraphics[width=\textwidth]{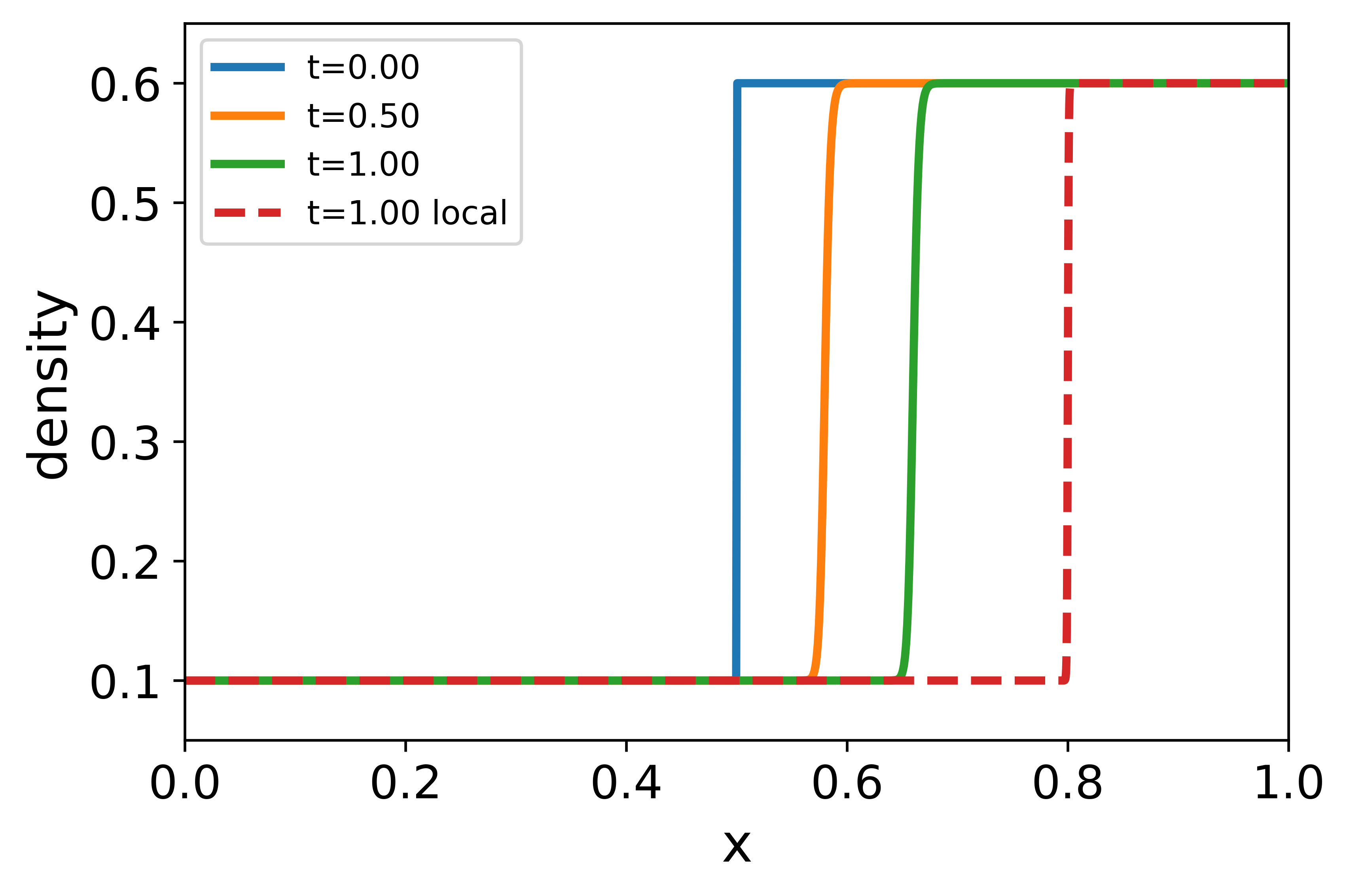}
	\end{subfigure}
	\begin{subfigure}{.32\textwidth}
	\includegraphics[width=\textwidth]{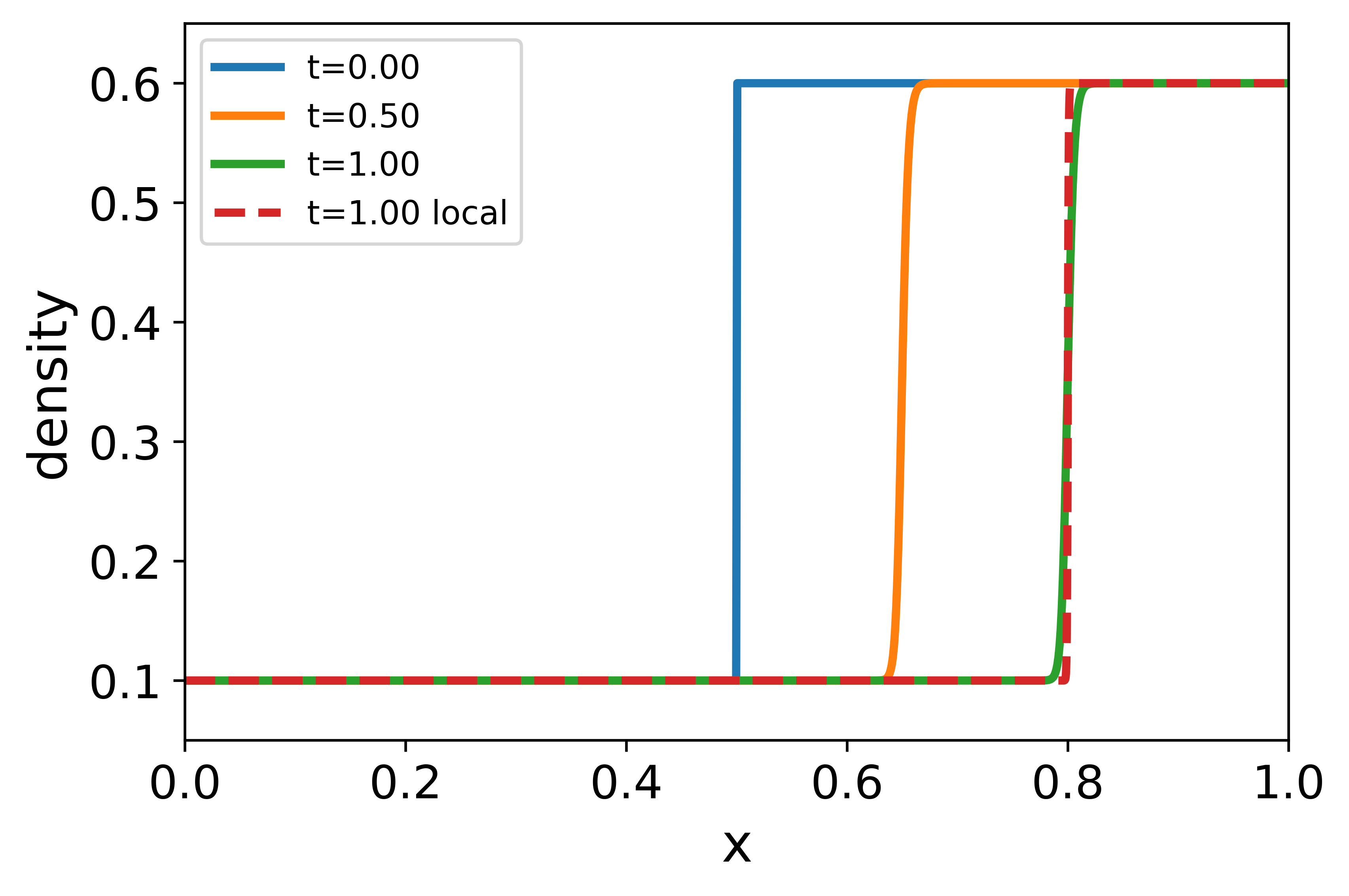}
	\end{subfigure}
	\begin{subfigure}{.32\textwidth}
	\includegraphics[width=\textwidth]{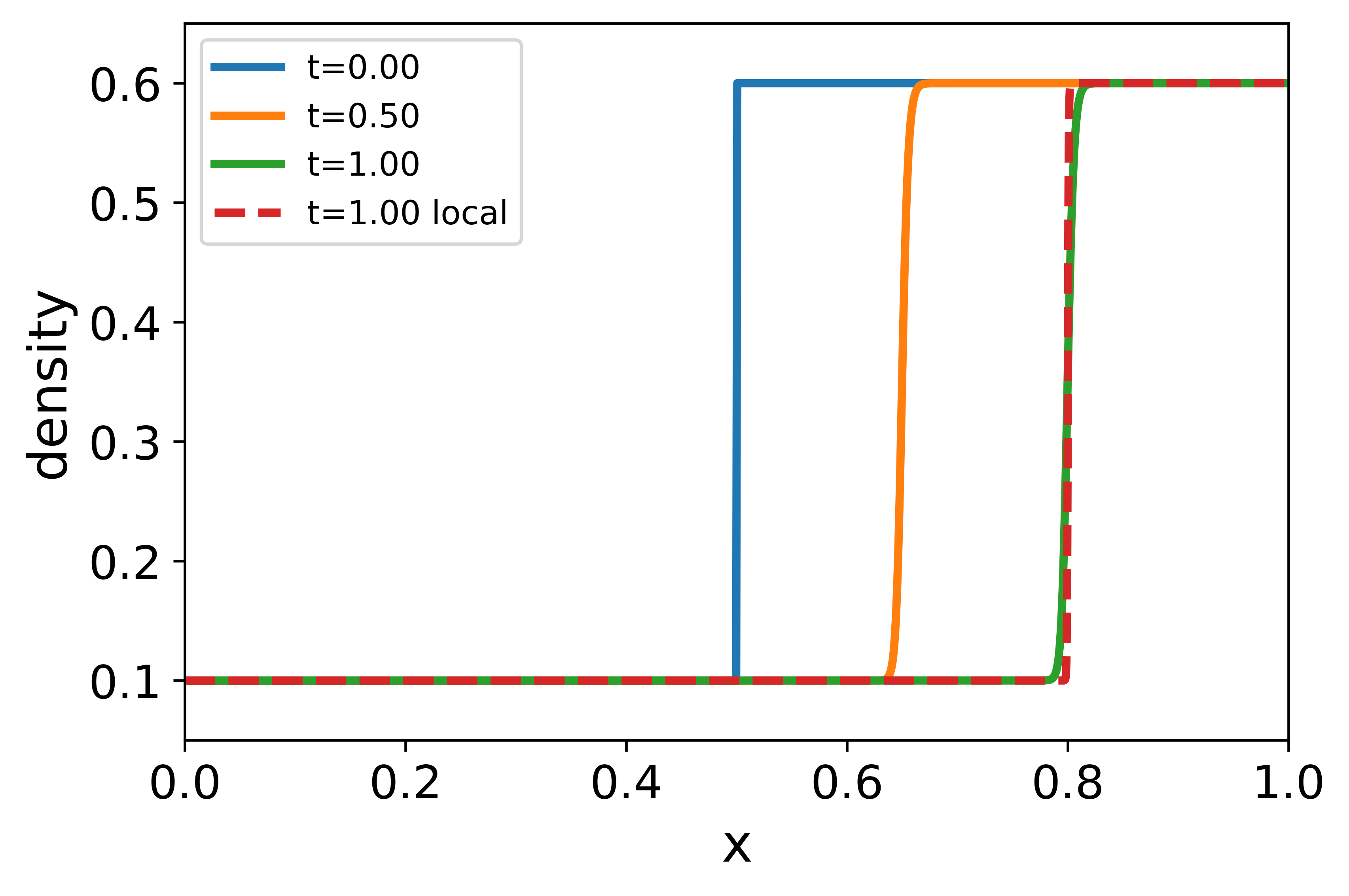}
	\end{subfigure}
	\caption{Experiment 1: Snapshots of computed solutions for the bell-shaped initial data (\emph{top}) and the Riemann initial data (\emph{bottom}) corresponding to the left endpoint quadrature weights (\emph{left}), the normalized left endpoint quadrature weights (\emph{middle}), and the exact quadrature weights (\emph{right}).}
	\label{fig:exp_1}
\end{figure}
For the bell-shaped initial data, we observe from the top row of Figure~\ref{fig:exp_1} that the numerical solutions of the nonlocal model preserve the smoothness of the initial data while the local solution develops a shock. At time $t=1$, the numerical solutions of the nonlocal model computed with the normalized left endpoint quadrature weights and the exact quadrature weights are close to the local solution, especially in the region away from the shock of the local solution.
This means that the numerical solution $\rho^{\delta,h}$ with both $\delta,h$ small provides a good approximation to the local solution $\rho^0$, which validates the conclusion of Theorem~\ref{thm:ac}.
We also observe from the top left figure of Figure~\ref{fig:exp_1} that the numerical solution of the nonlocal model computed with the left endpoint quadrature weights is very different from the local solution at time $t=1$. Although the numerical solution of the nonlocal model still approximates a shock profile at time $t=1$, the shock position is incorrect.
The comparison between the three sets of numerical quadrature weights emphasizes the significance of the normalization condition \eqref{eq:normalization_condition} for numerical quadrature weights.

For the Riemann initial data, the local solution $\rho^0$ is a traveling wave moving at the constant speed $v=1-(\rho_L+\rho_R)=0.3$. 
We observe from the bottom row of Figure~\ref{fig:exp_1} that the numerical solutions of the nonlocal model computed with the normalized left endpoint quadrature weights and the exact quadrature weights are close to the local solution at time $t=1$.
Meanwhile, in contrast to the discontinuity of the local solution, the nonlocal solutions get smoothed because of the nonlocal effects.
We also observe from the bottom left figure of Figure~\ref{fig:exp_1} that the numerical solution of the nonlocal model computed with the left endpoint quadrature weights is very different from the local solution at time $t=1$. While the former still approximates a Riemann data at time $t=1$, the position of the jump from $\rho_L=0.1$ to $\rho_R=0.6$ is incorrect.
The comparison again emphasizes the significance of the normalization condition \eqref{eq:normalization_condition} for numerical quadrature weights.

{\bf Experiment 2.} We next check the asymptotically compatibility of the scheme \eqref{eq:nonlocal_lwr_num}-\eqref{eq:nonlocal_lwr_num_2} by plotting $\norm{\rho^{\delta,h} - \rho^0}_{\mathbf{L}^1}$ with $\delta \propto h \to 0$.
We take $\delta=mh$ where $m=1,2,5$ and $h=0.01\times 2^{-l}$ for $l=0,1,2,3$, and compute numerical solutions $\rho^{\delta,h}$ using the Lax-Friedrichs numerical flux function \eqref{eq:g_func_LxF} and different numerical quadrature weights.
The local solution $\rho^0$ is numerically solved on a fine grid with $h=0.01\times 2^{-5}$ using a Lax-Friedrichs scheme for \eqref{eq:lwr} with the numerical flux function \eqref{eq:local_LxF}.
For each initial data and each set of numerical quadrature weights, we compute the $\mathbf{L}^1$ error $\norm{\rho^{\delta,h} - \rho^0}_{\mathbf{L}^1}$ with an interpolation of $\rho^{\delta,h}$ onto the fine grid on which $\rho^0$ is computed, and plot $\norm{\rho^{\delta,h} - \rho^0}_{\mathbf{L}^1}$ against $h^{-1}$ in the log-log scale for $\delta=h$, $\delta=2h$, and $\delta=5h$ in different colors.
We also plot a dashed line with the slope $-1$ to represent the linear convergence rate.
See the results in Figure~\ref{fig:exp_2}.

\begin{figure}[htbp]
\centering
	\begin{subfigure}{.32\textwidth}
	\includegraphics[width=\textwidth]{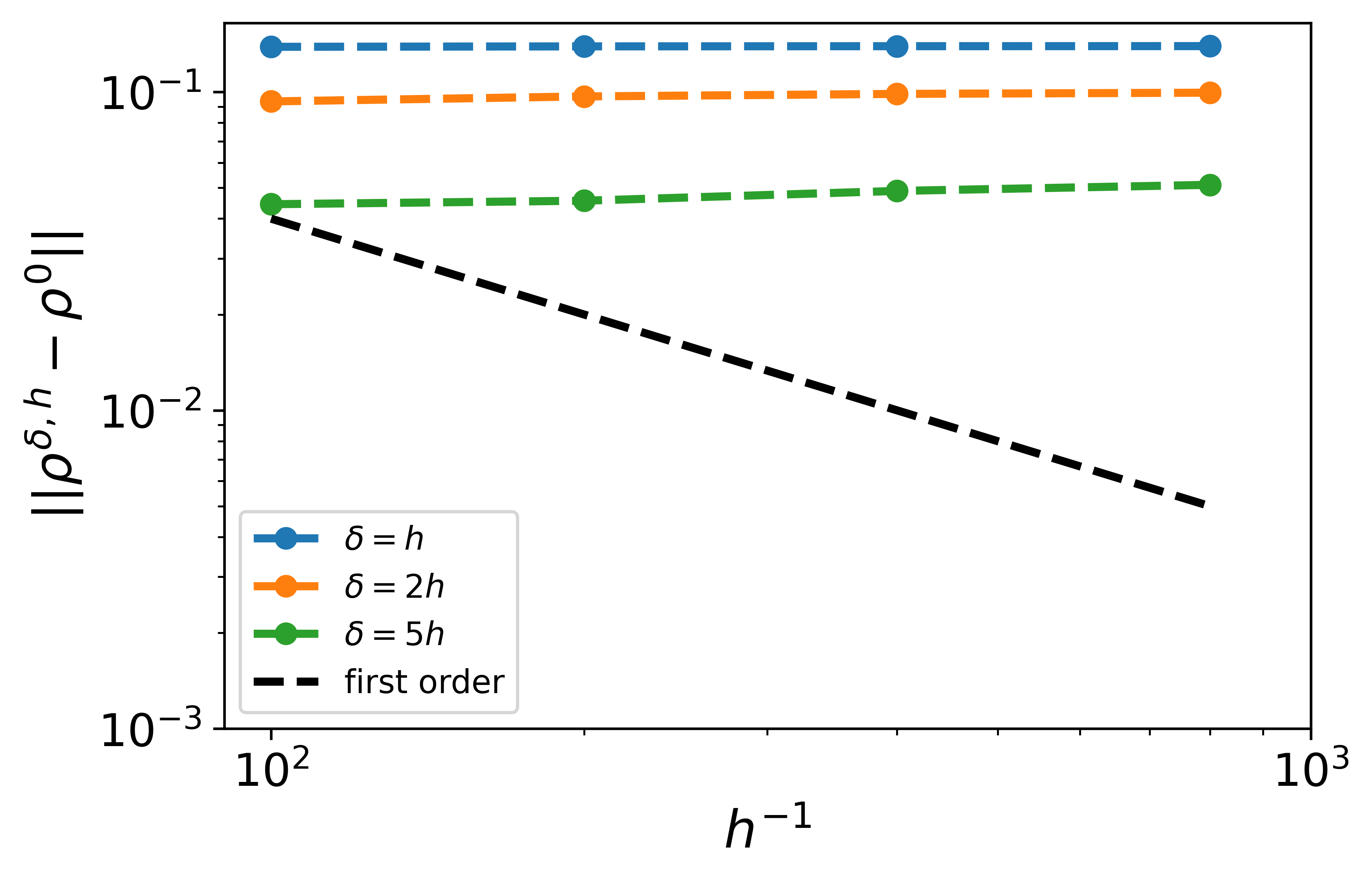}
	\end{subfigure}
	\begin{subfigure}{.32\textwidth}
	\includegraphics[width=\textwidth]{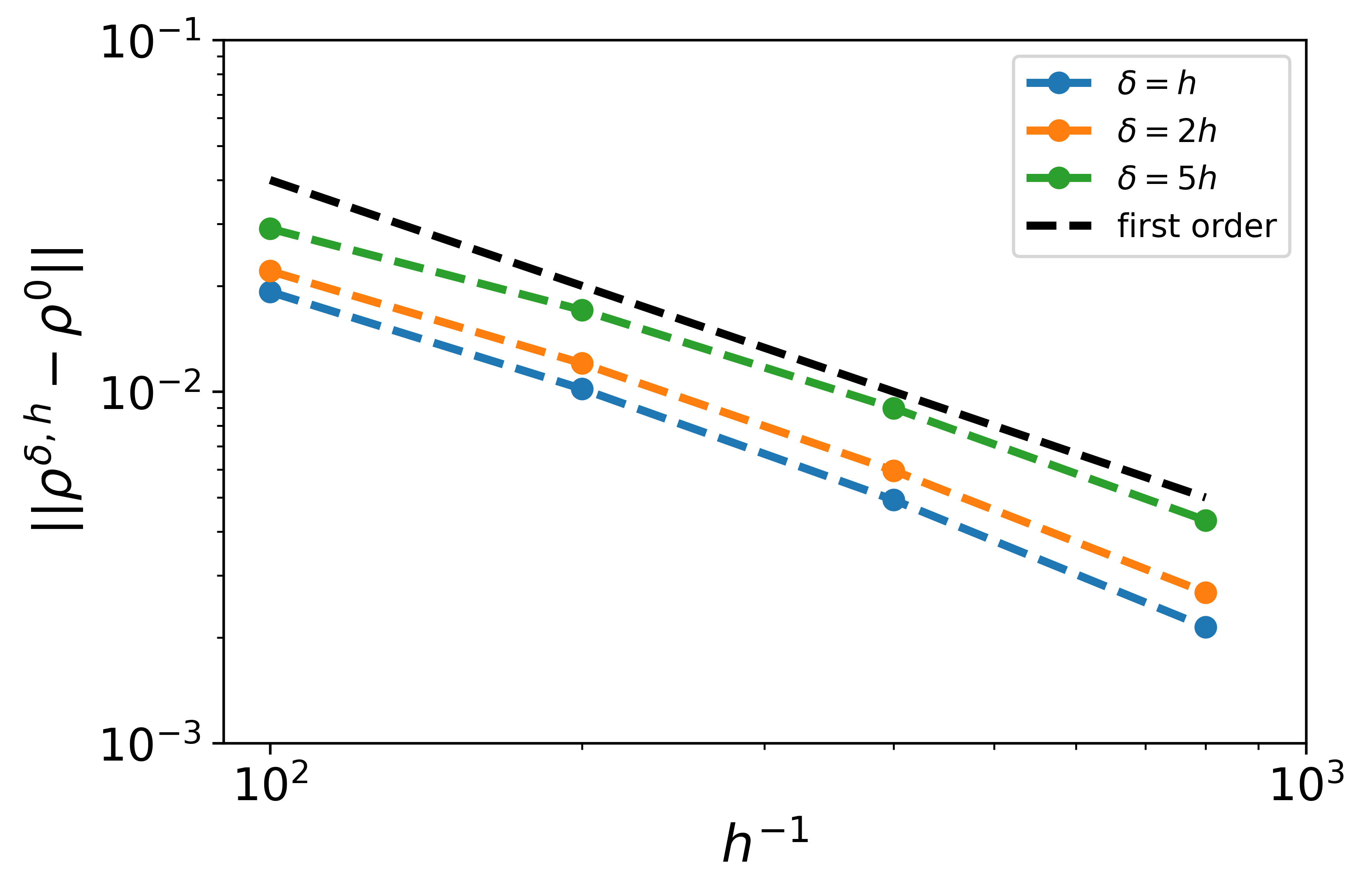}
	\end{subfigure}
	\begin{subfigure}{.32\textwidth}
	\includegraphics[width=\textwidth]{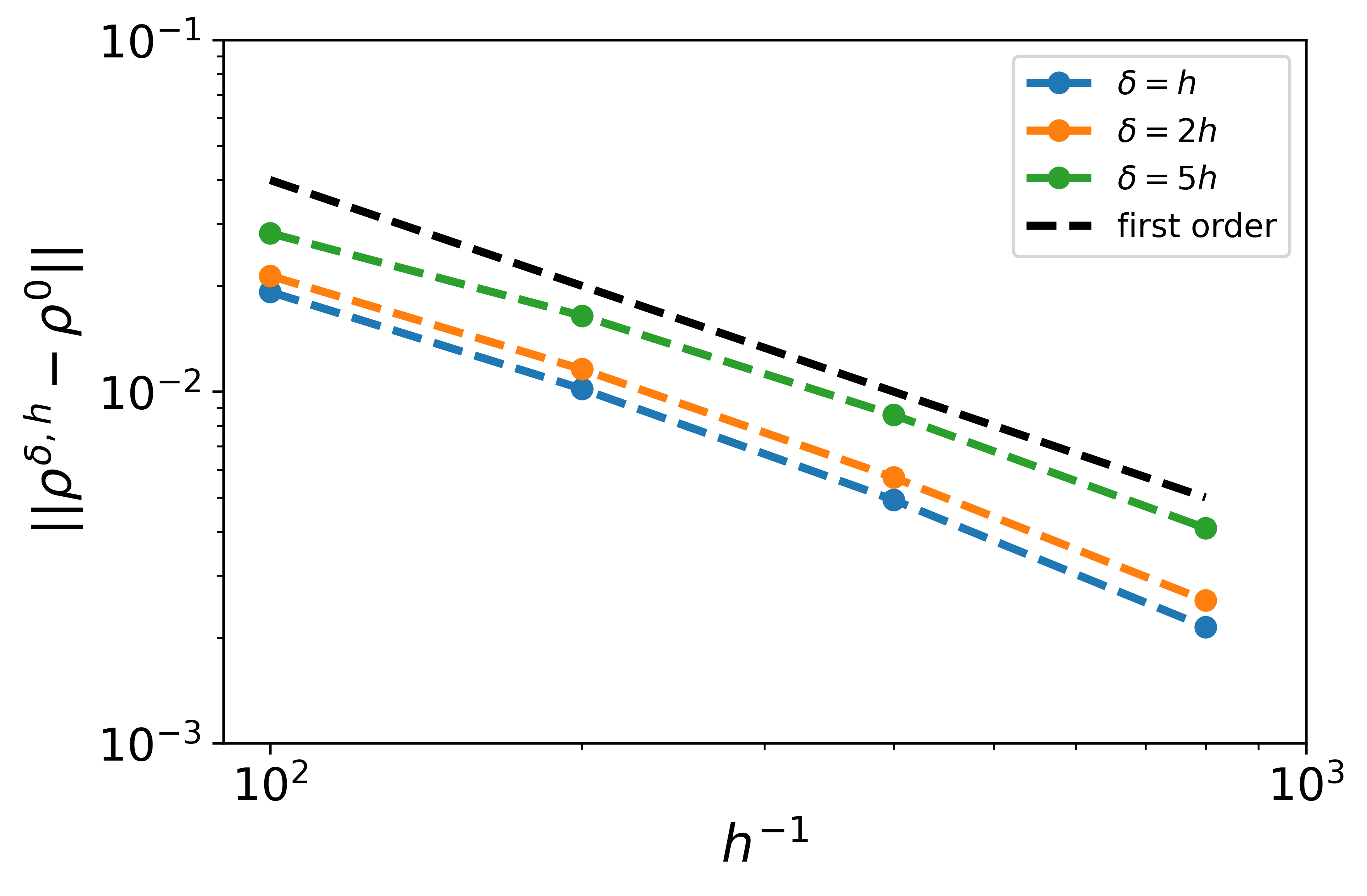}
	\end{subfigure}
	\begin{subfigure}{.32\textwidth}
	\includegraphics[width=\textwidth]{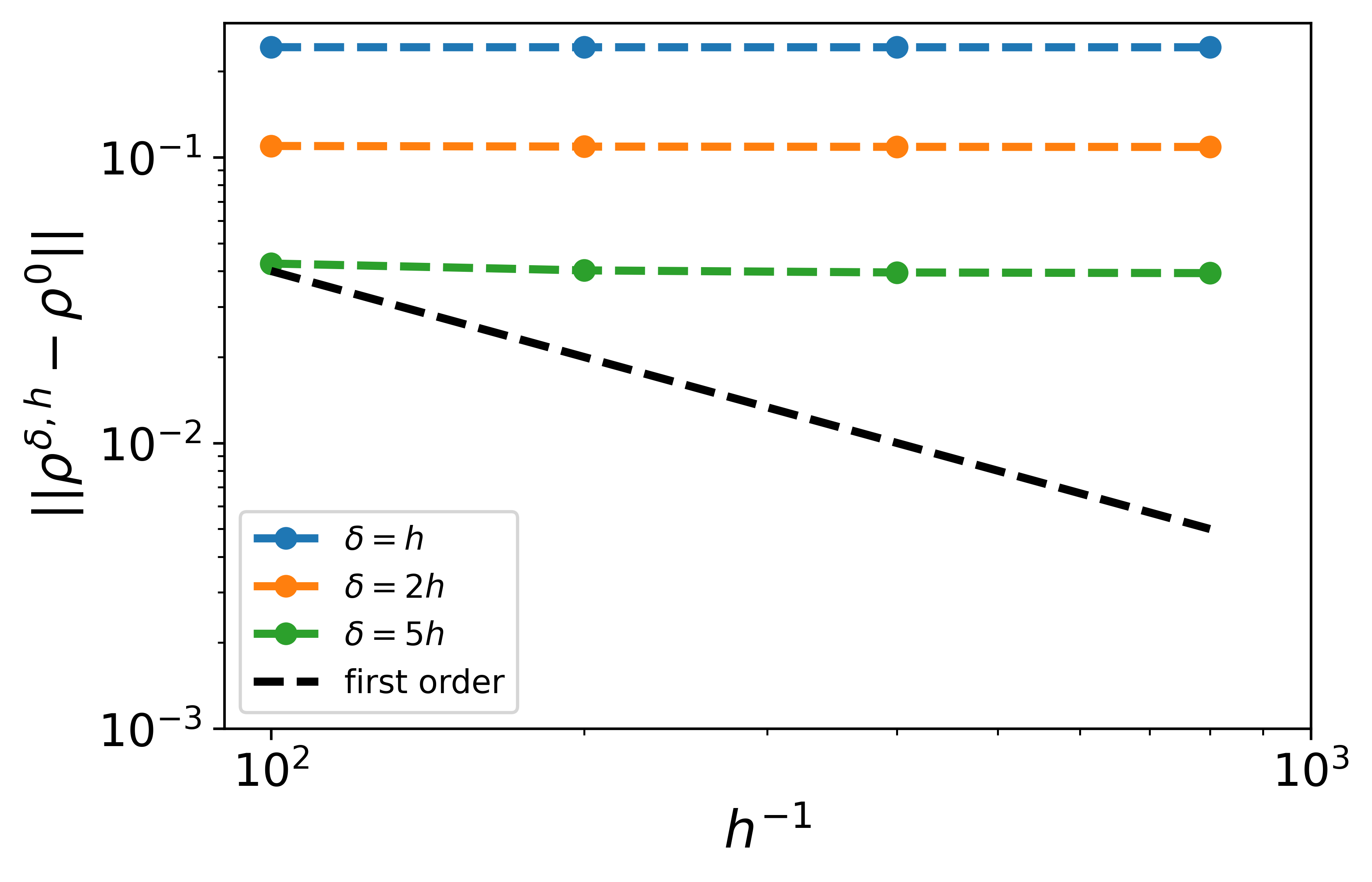}
	\end{subfigure}
	\begin{subfigure}{.32\textwidth}
	\includegraphics[width=\textwidth]{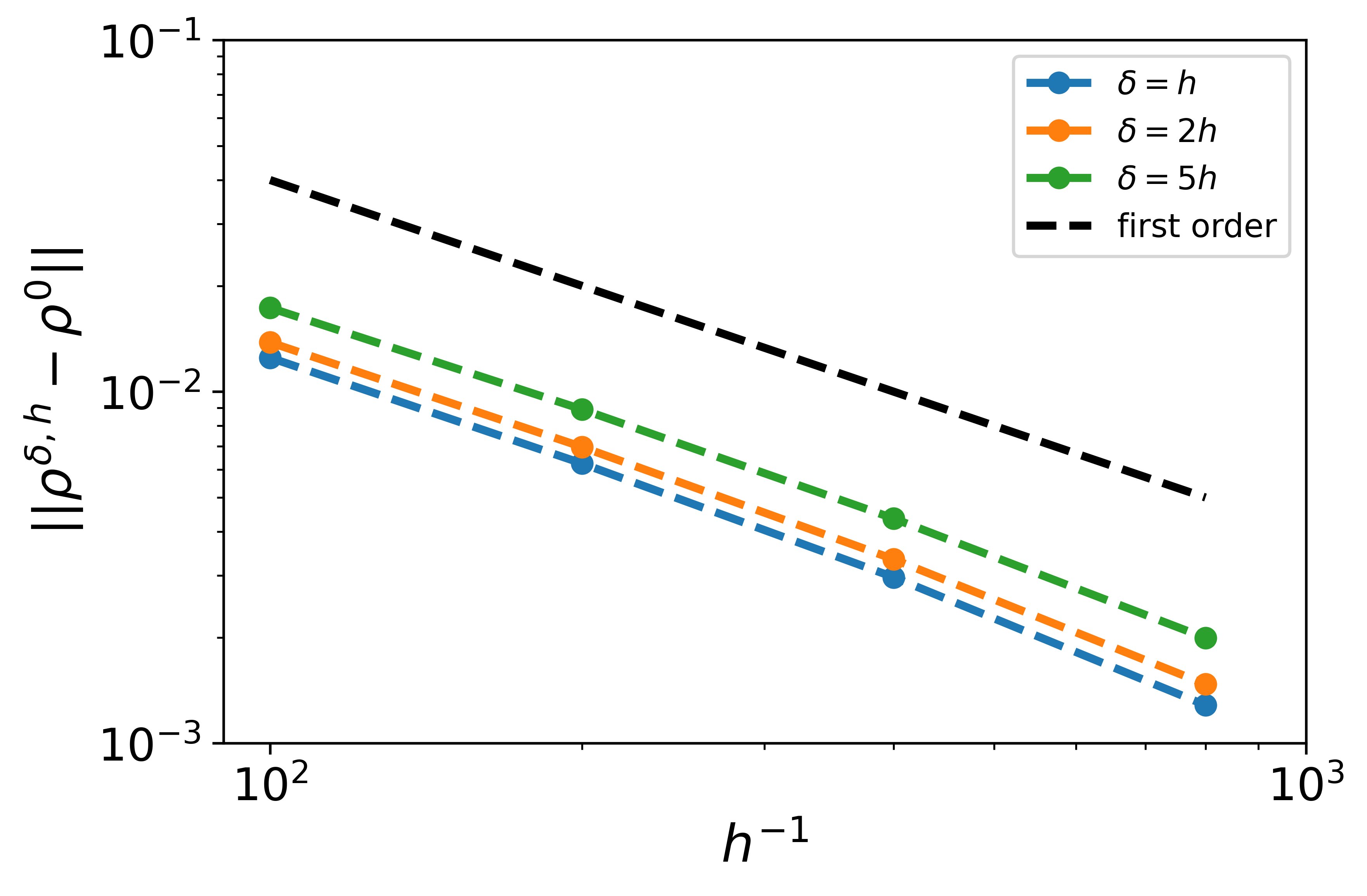}
	\end{subfigure}
	\begin{subfigure}{.32\textwidth}
	\includegraphics[width=\textwidth]{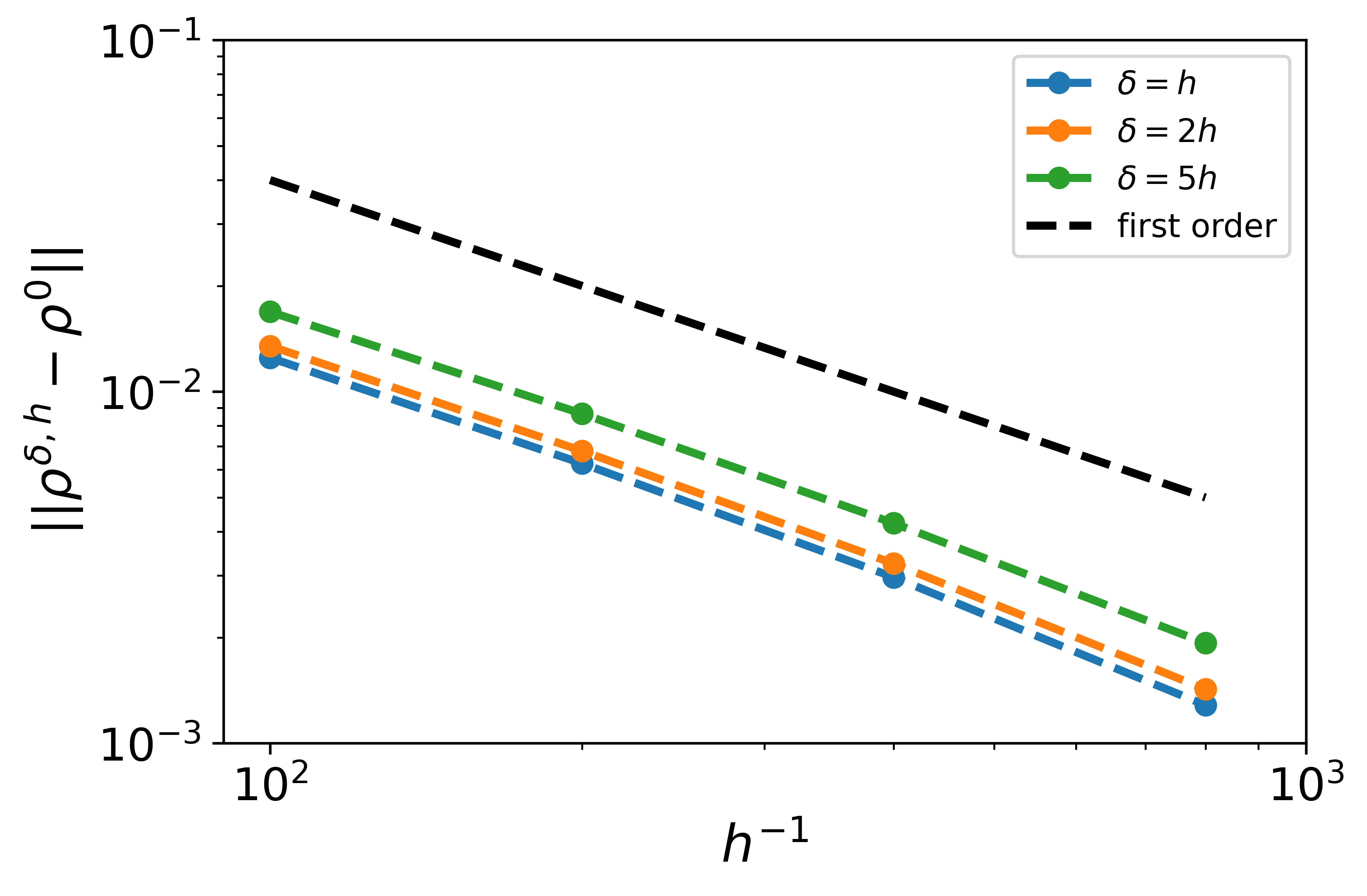}
	\end{subfigure}
	\caption{Experiment 2: Convergence from $\rho^{\delta,h}$ to $\rho^0$ for the bell-shaped initial data (\emph{top}) and the Riemann initial data (\emph{bottom}) corresponding to the left endpoint quadrature weights (\emph{left}), the normalized left endpoint quadrature weights (\emph{middle}), and the exact quadrature weights (\emph{right}).}
	\label{fig:exp_2}
\end{figure}

We observe from Figure~\ref{fig:exp_2} that: for the normalized left endpoint quadrature weights and the exact quadrature weights, the error $\norm{\rho^{\delta,h} - \rho^0}_{\mathbf{L}^1}$ has a linear decay rate with respect to $h$ for both initial data and $\delta=mh$ for $m=1,2,5$. This means that $\rho^{\delta,h}$ converges to $\rho^0$ along the limiting paths $\delta=mh\to0$ for $m=1,2,5$, which validates the conclusion of Theorem~\ref{thm:ac}.
Moreover, the numerical results show that the convergence is of first order with the particular choices of the initial data and the limiting paths.
In contrast, for the left endpoint numerical quadrature weights, the error $\norm{\rho^{\delta,h} - \rho^0}_{\mathbf{L}^1}$ stagnates on the scale of $10^{-1}$ for both initial data and $\delta=mh$ for $m=1,2,5$.
This is due to the convergence of $\rho^{\delta,h}$ to an incorrect solution when $\delta=mh\to 0$, further highlighting the importance of asymptotically compatibility via
 the normalization condition \eqref{eq:normalization_condition}.

{\bf Experiment 3.}  We now check the uniform convergence of the scheme \eqref{eq:nonlocal_lwr_num}-\eqref{eq:nonlocal_lwr_num_2} with respect to $\delta$ by plotting $\norm{\rho^{\delta,h} - \rho^\delta}_{\mathbf{L}^1}$ with $h \to 0$ for different choices of $\delta$.
We take $\delta=0.01\times 2^{-l}$ for $l=0,1,2$ and $h=0.01\times 2^{-l}$ for $l=0,1,2,3$, and compute numerical solutions $\rho^{\delta,h}$ using the Lax-Friedrichs numerical flux function \eqref{eq:g_func_LxF} and different numerical quadrature weights.
The reference solution $\rho^\delta$ is numerically solved on a fine grid with $h=0.01\times 2^{-5}$ using the same scheme.
For each initial data and each set of numerical quadrature weights, we compute the $\mathbf{L}^1$ error $\norm{\rho^{\delta,h} - \rho^\delta}_{\mathbf{L}^1}$ with an interpolation of $\rho^{\delta,h}$ onto the fine grid, and plot $\norm{\rho^{\delta,h} - \rho^\delta}_{\mathbf{L}^1}$ with respect to $h^{-1}$ in the log-log scale for $\delta=0.01$, $\delta=0.005$, and $\delta=0.0025$ in different colors.
A dashed line with the slope $-1$ is again provided.
See the results in Figure~\ref{fig:exp_3}.

\begin{figure}[htbp]
\centering
	\begin{subfigure}{.32\textwidth}
	\includegraphics[width=\textwidth]{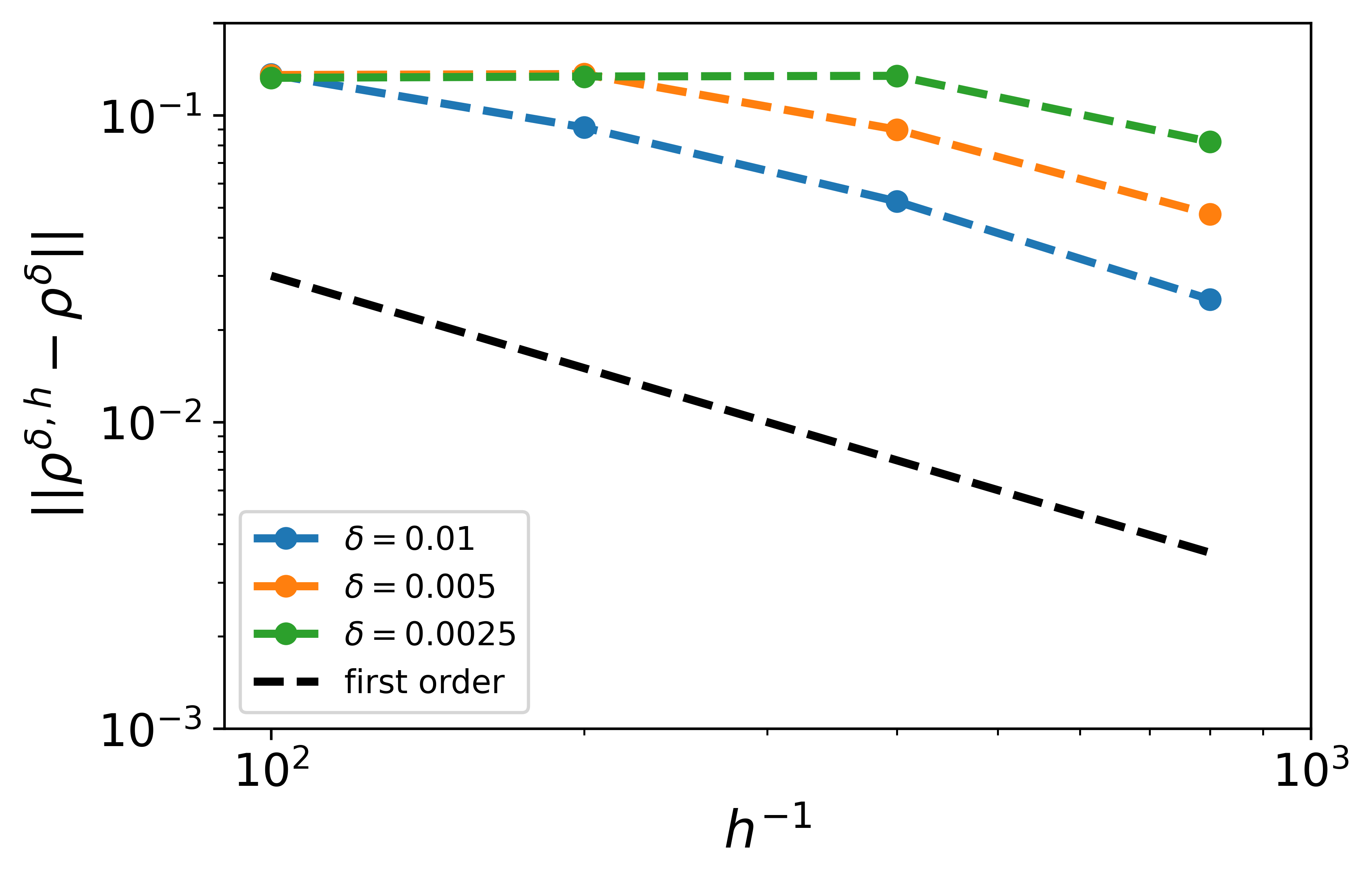}
	\end{subfigure}
	\begin{subfigure}{.32\textwidth}
	\includegraphics[width=\textwidth]{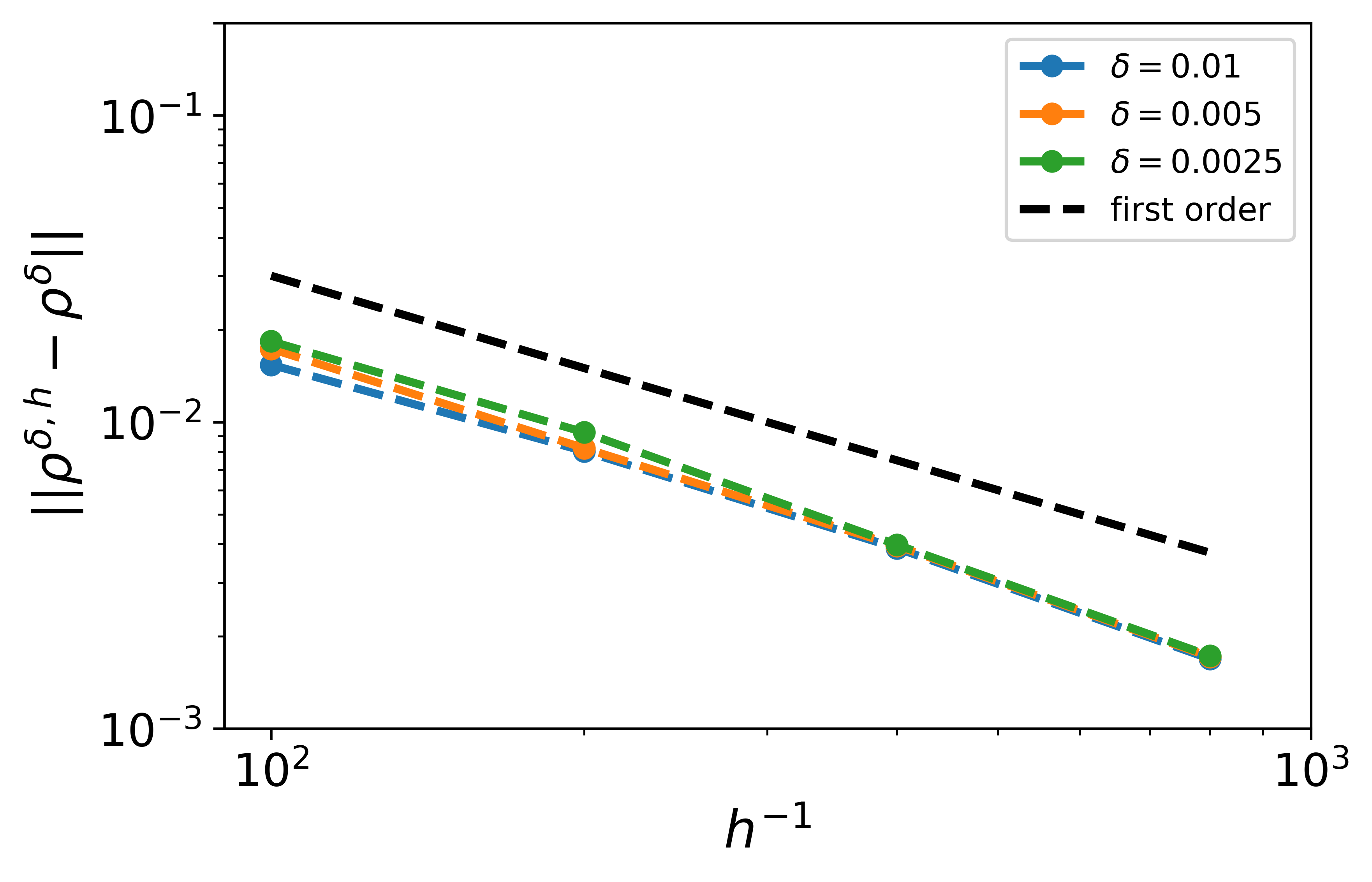}
	\end{subfigure}
	\begin{subfigure}{.32\textwidth}
	\includegraphics[width=\textwidth]{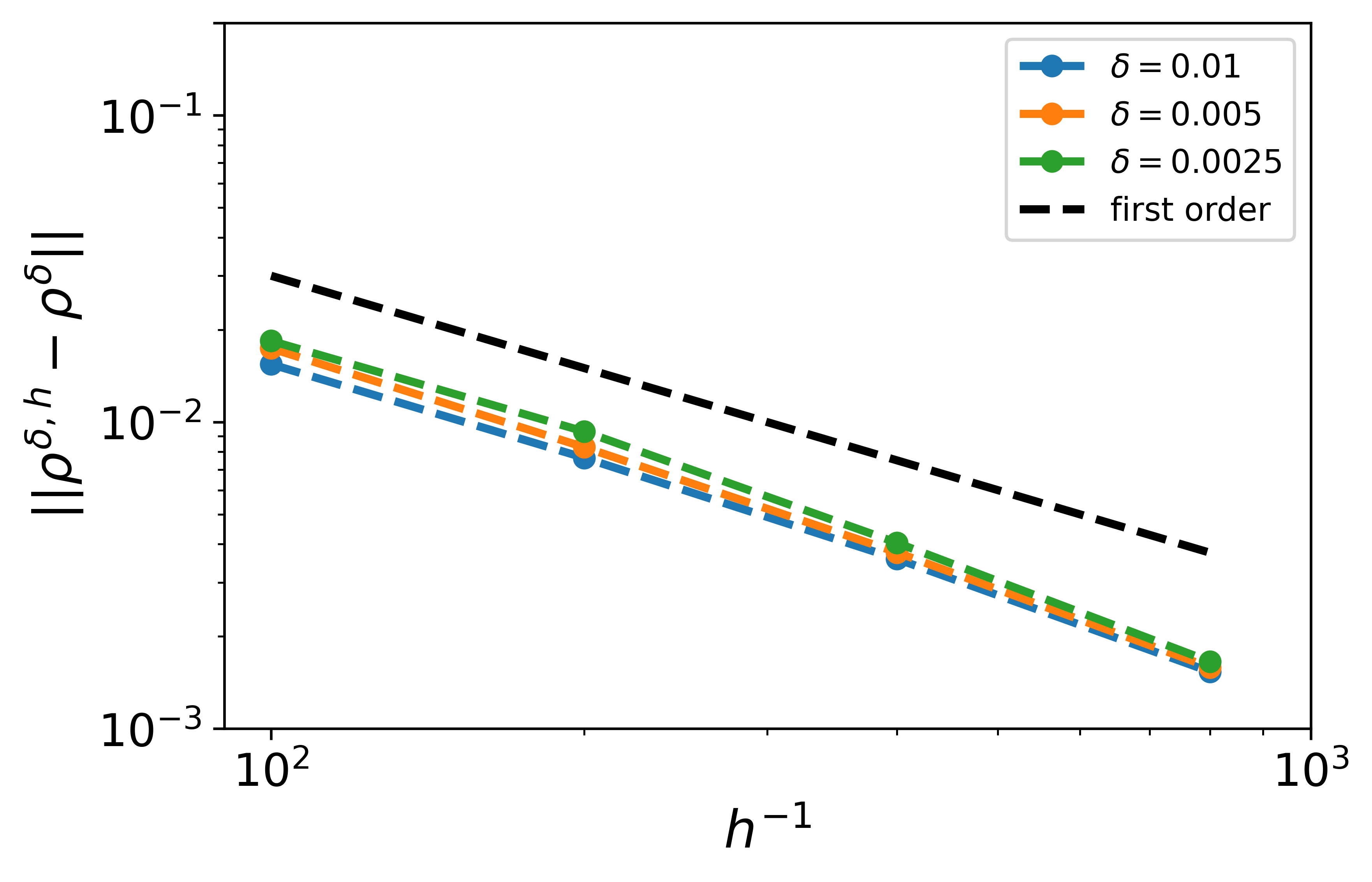}
	\end{subfigure}
	\begin{subfigure}{.32\textwidth}
	\includegraphics[width=\textwidth]{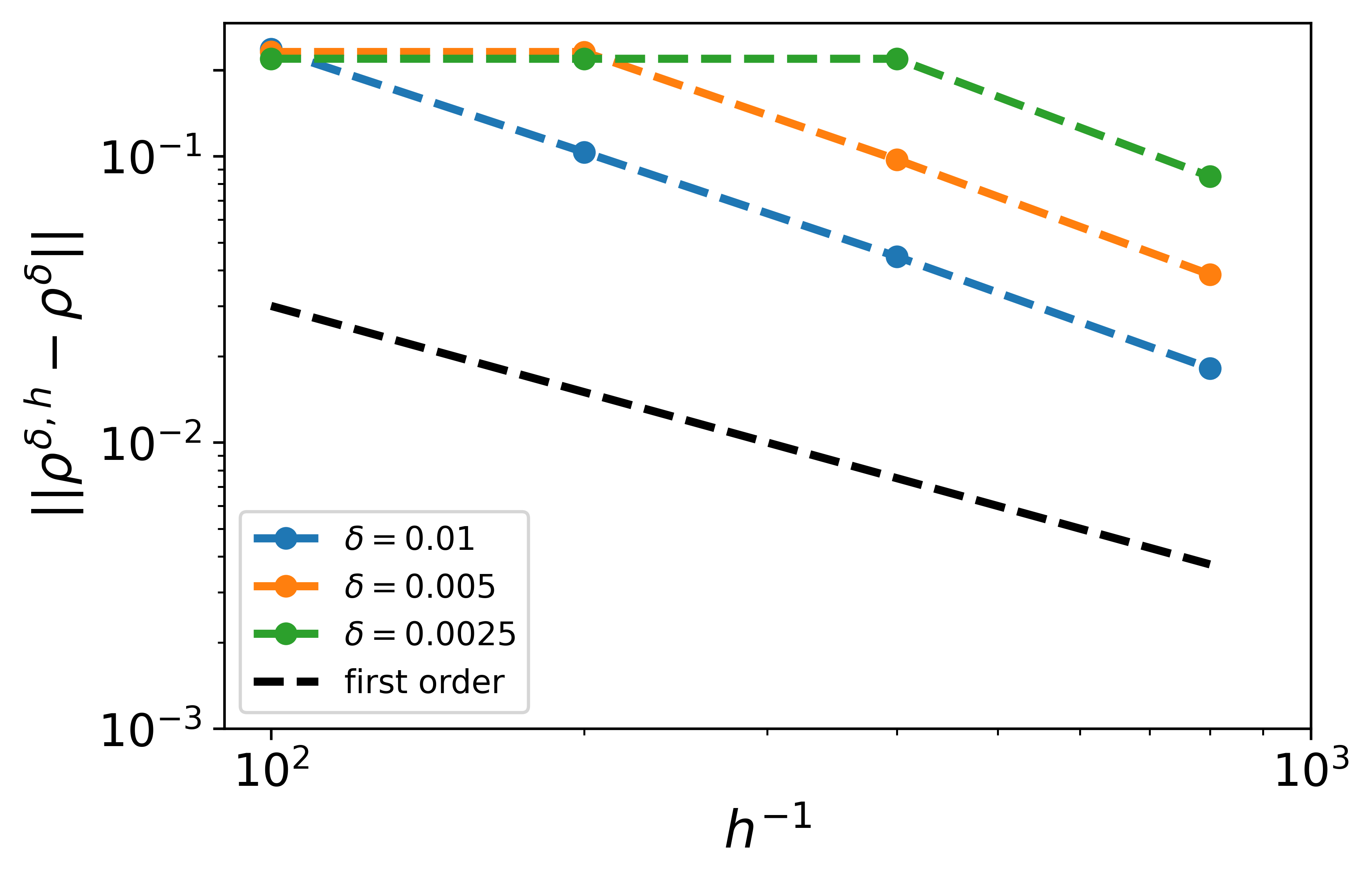}
	\end{subfigure}
	\begin{subfigure}{.32\textwidth}
	\includegraphics[width=\textwidth]{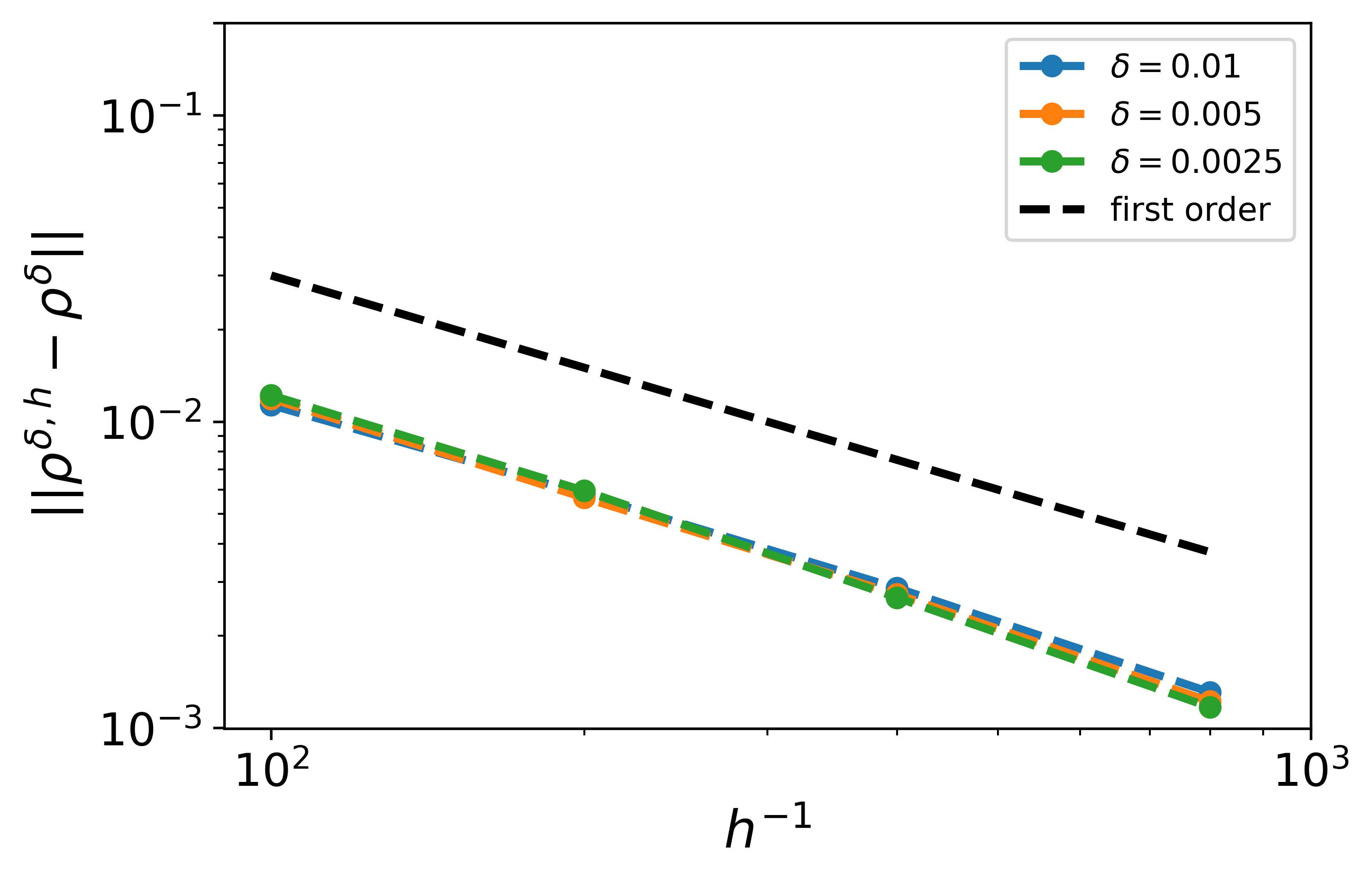}
	\end{subfigure}
	\begin{subfigure}{.32\textwidth}
	\includegraphics[width=\textwidth]{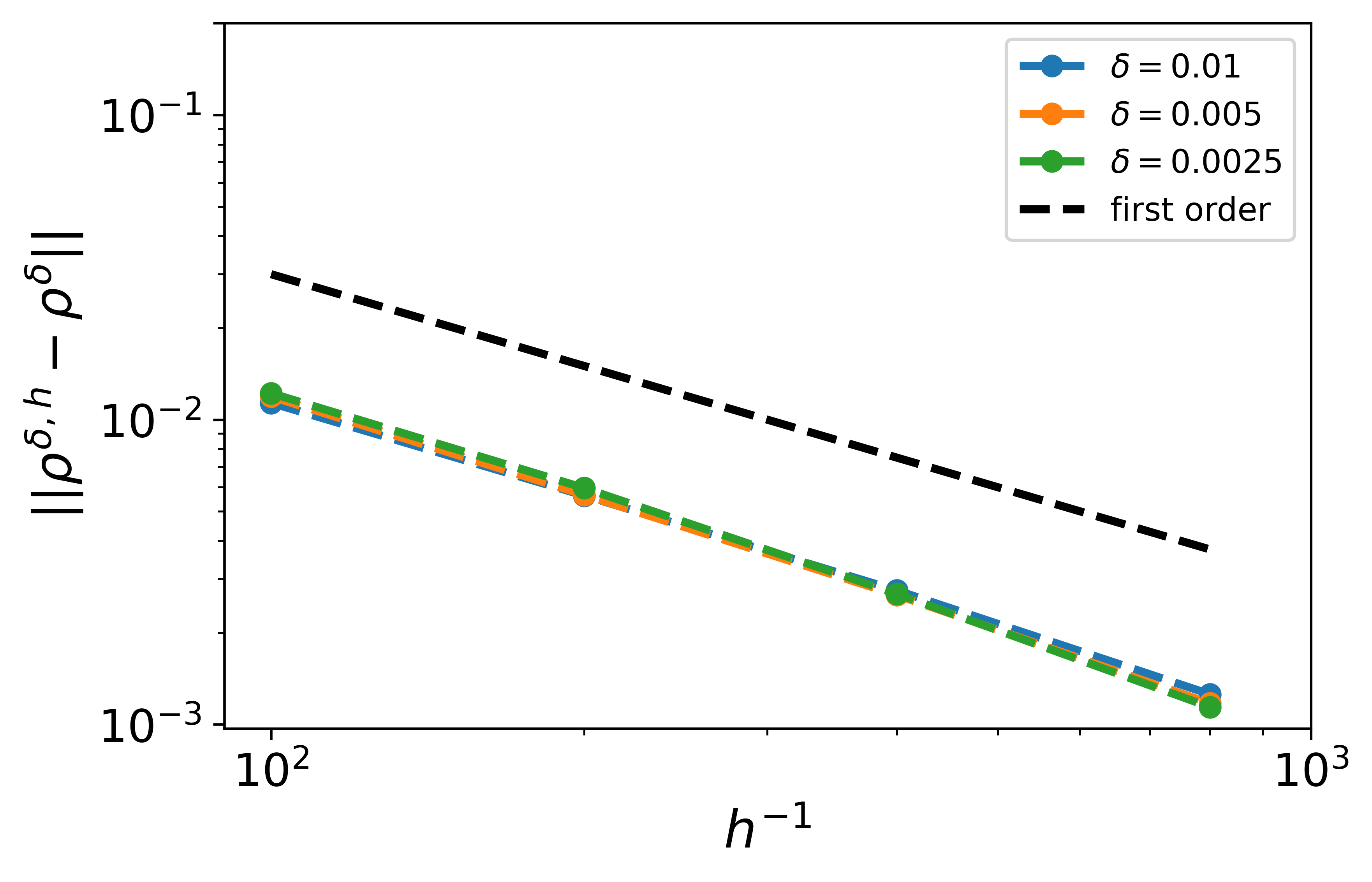}
	\end{subfigure}
	\caption{Experiment 3: Convergence from $\rho^{\delta,h}$ to $\rho^\delta$ for the bell-shaped initial data (\emph{top}) and the Riemann initial data (\emph{bottom}) corresponding to the left endpoint quadrature weights (\emph{left}), the normalized left endpoint quadrature weights (\emph{middle}), and the exact quadrature weights (\emph{right}).}
	\label{fig:exp_3}
\end{figure}
From Figure~\ref{fig:exp_3}, we see that for the normalized left endpoint quadrature weights and the exact quadrature weights, the error $\norm{\rho^{\delta,h} - \rho^\delta}_{\mathbf{L}^1}$ has a linear decay rate with respect to $h$ for both initial data and all choices of $\delta$.
Moreover, the plots of $\norm{\rho^{\delta,h} - \rho^\delta}_{\mathbf{L}^1}$ with respect to $h^{-1}$ have very little change for $\delta=0.01$, $\delta=0.005$, and $\delta=0.0025$. This means that $\rho^{\delta,h}$ converges to $\rho^\delta$ as $h\to0$ uniformly in $\delta$, which validates the conclusion of Theorem~\ref{thm:numerical_convergence}.
In addition, the numerical results show that the convergence is of first order with the particular choices of the initial data and the parameter $\delta$.
In contrast, for the left endpoint numerical quadrature weights, the error $\norm{\rho^{\delta,h} - \rho^\delta}_{\mathbf{L}^1}$ stagnates on the scale of $10^{-1}$ when $h\geq\delta$ for both initial data and all choices of $\delta$.
This may because $\rho^\delta$ approximates $\rho^0$ well when $\delta$ is small while $\rho^{\delta,h}=\rho^{0,h}$ when $h\geq\delta$ and $\rho^{0,h}$ is not a consistent numerical approximation to $\rho^0$.
We also observe that, in each case, the error decays when $h<\delta$. However, the error increases when $\delta$ decreases from $0.01$ to $0.0025$ for any fixed mesh size $h$. One can infer that the convergence from $\rho^{\delta,h}$ to $\rho^\delta$ as $h\to0$ will become slower and slower as $\delta\to0$, and the uniform convergence cannot hold, which is again showing the importance of 
 the normalization condition \eqref{eq:normalization_condition} for the uniform convergence of the scheme \eqref{eq:nonlocal_lwr_num}-\eqref{eq:nonlocal_lwr_num_2}.

{\bf Experiment 4.} We finally test the scheme \eqref{eq:nonlocal_lwr_num}-\eqref{eq:nonlocal_lwr_num_2} with different choices of the nonlocal kernel.
Besides the linear decreasing kernel considered before, we also use the exponential kernel $w_\delta(s)=\frac{e^{-\frac{s}{\delta}}}{\delta(1-e^{-1})}$
and the constant kernel
$
	w_\delta(s) = \frac1\delta
$, and adopt the exact quadrature weights \eqref{eq:quad_weight_exact}.
We take $\delta=mh$ where $m=1,2,5$ and $h=0.01\times 2^{-l}$ for $l=0,1,2,3$, and compute numerical solutions $\rho^{\delta,h}$ using the Lax-Friedrichs numerical flux function \eqref{eq:g_func_LxF} and different numerical quadrature weights.
The local solution $\rho^0$ is numerically solved on a fine grid with $h=0.01\times 2^{-5}$ using a Lax-Friedrichs scheme for \eqref{eq:lwr} with the numerical flux function \eqref{eq:local_LxF}.
For each initial data and each nonlocal kernel, we compute the $\mathbf{L}^1$ error $\norm{\rho^{\delta,h} - \rho^0}_{\mathbf{L}^1}$.
A dashed line with the slope $-1$ is again provided.
See the results in Figure~\ref{fig:exp_4}.

\begin{figure}[htbp]
\centering
	\begin{subfigure}{.32\textwidth}
	\includegraphics[width=\textwidth]{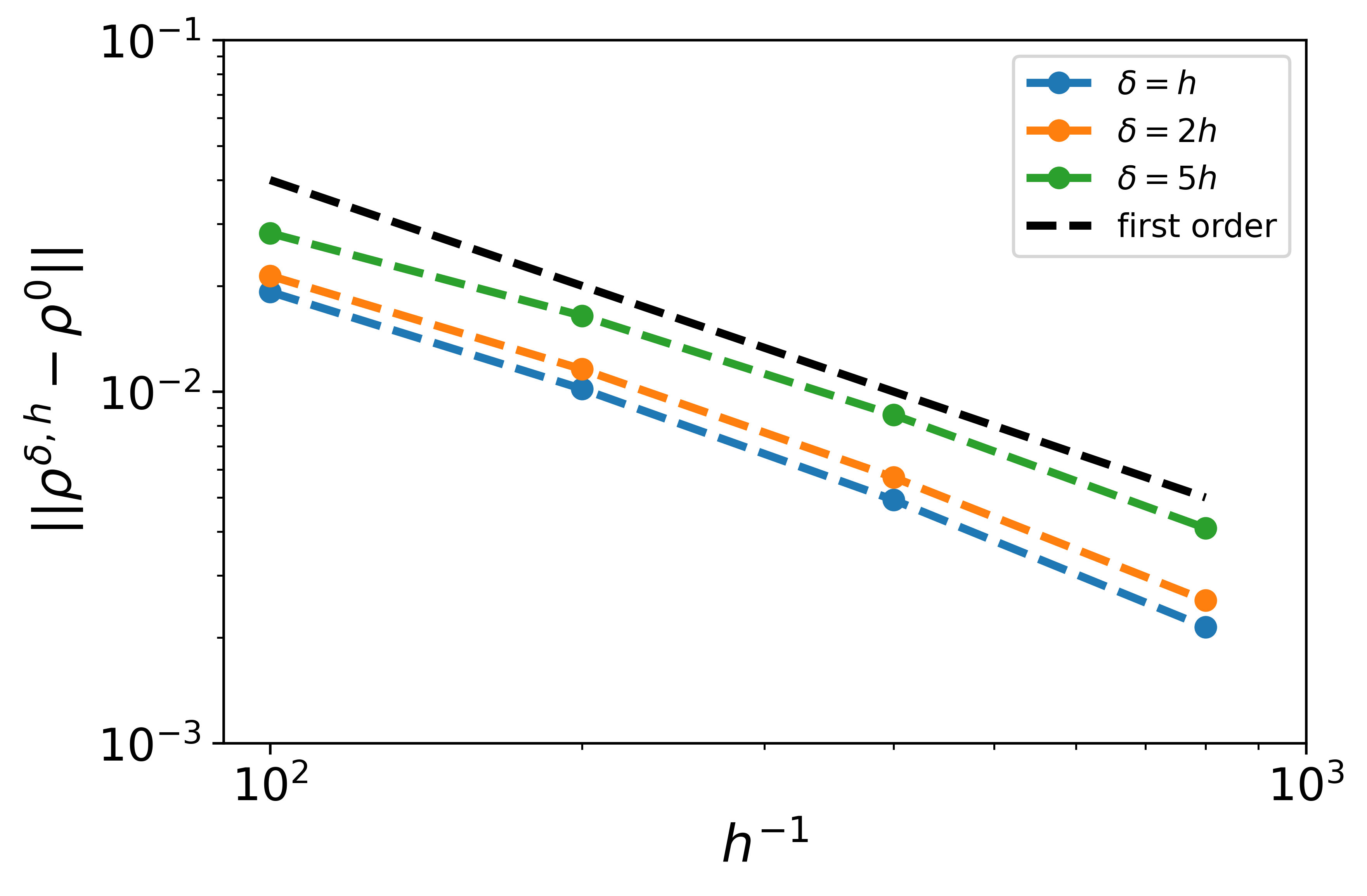}
	\end{subfigure}
	\begin{subfigure}{.32\textwidth}
	\includegraphics[width=\textwidth]{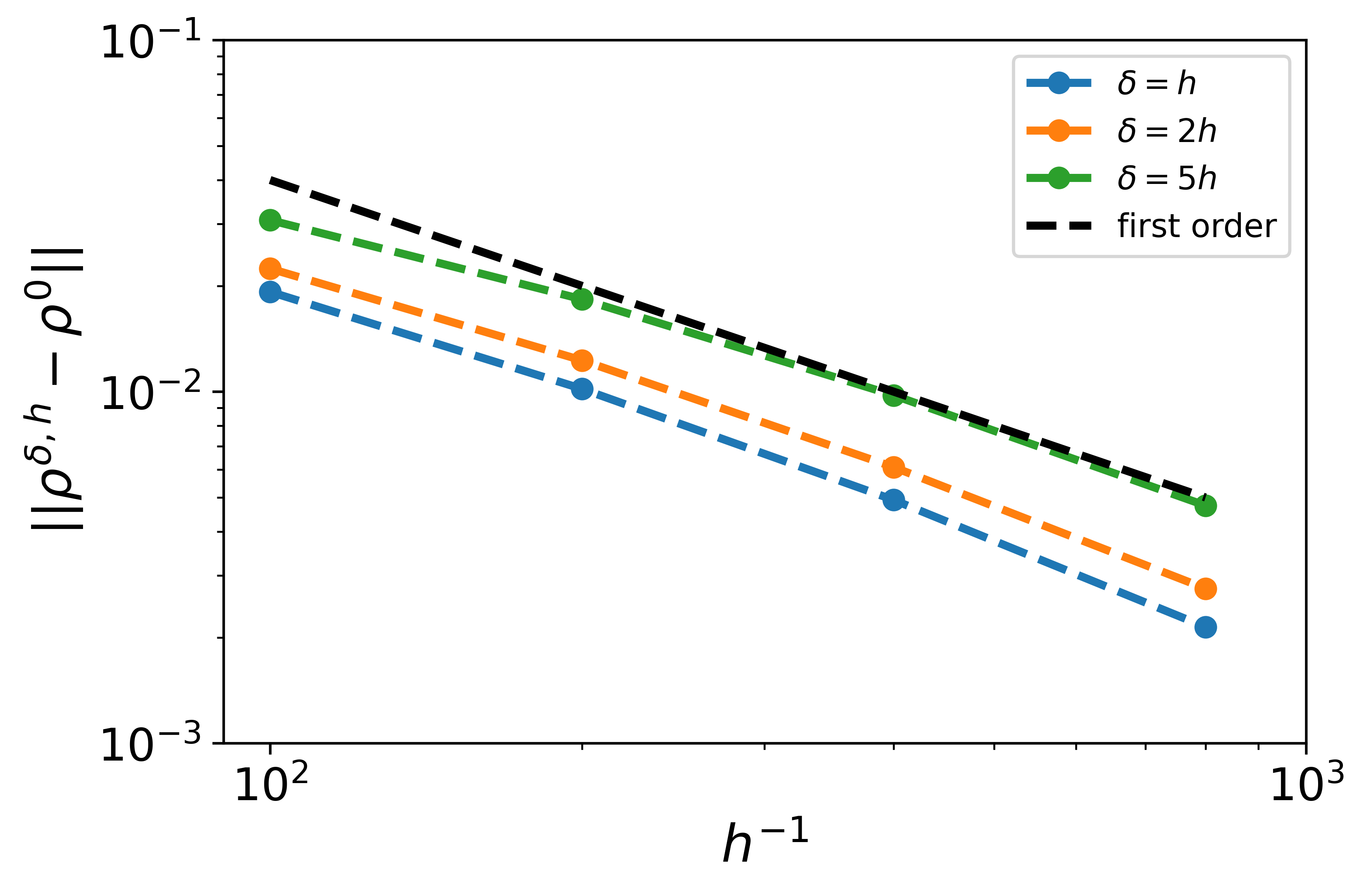}
	\end{subfigure}
	\begin{subfigure}{.32\textwidth}
	\includegraphics[width=\textwidth]{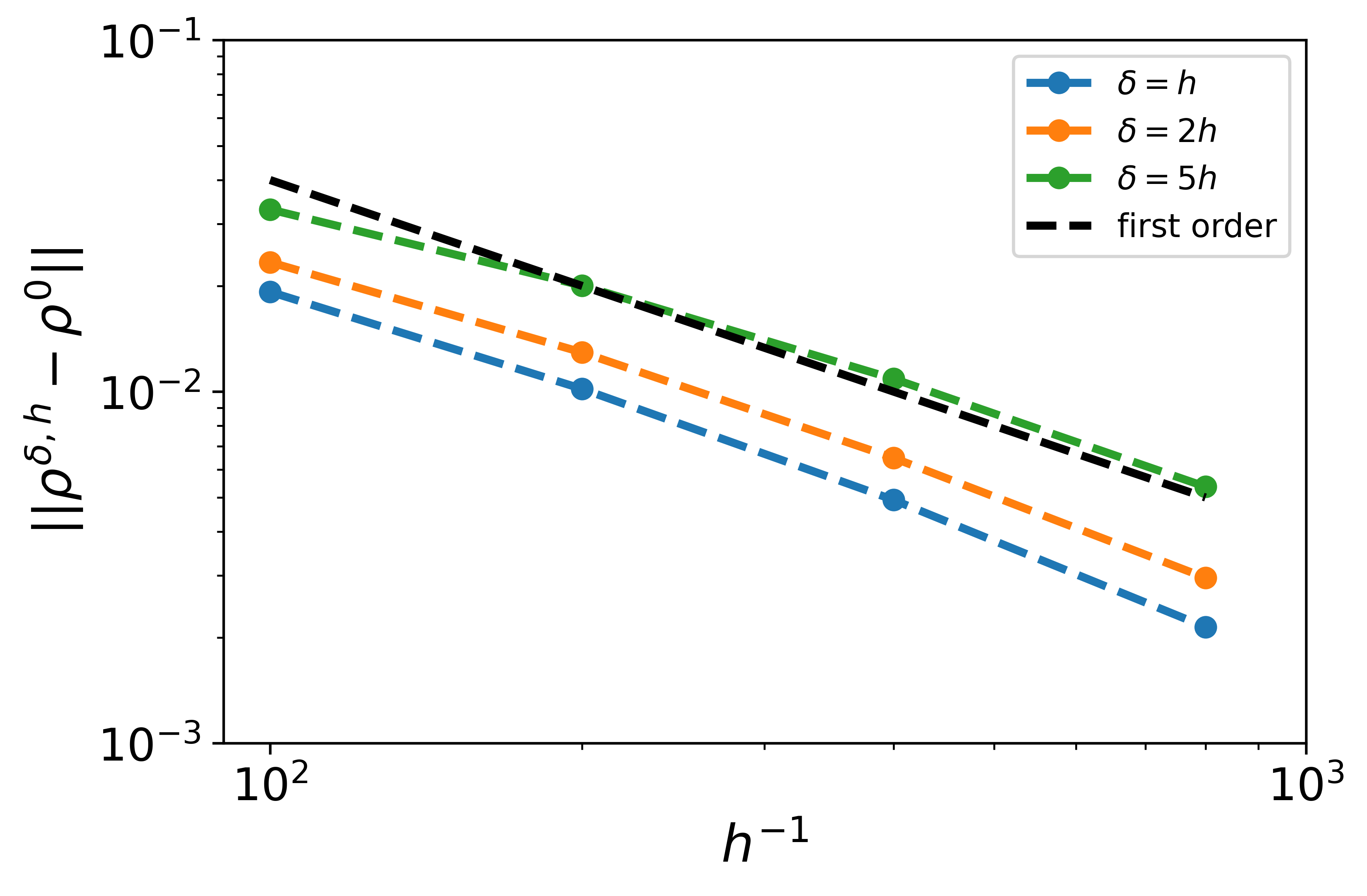}
	\end{subfigure}
	\begin{subfigure}{.32\textwidth}
	\includegraphics[width=\textwidth]{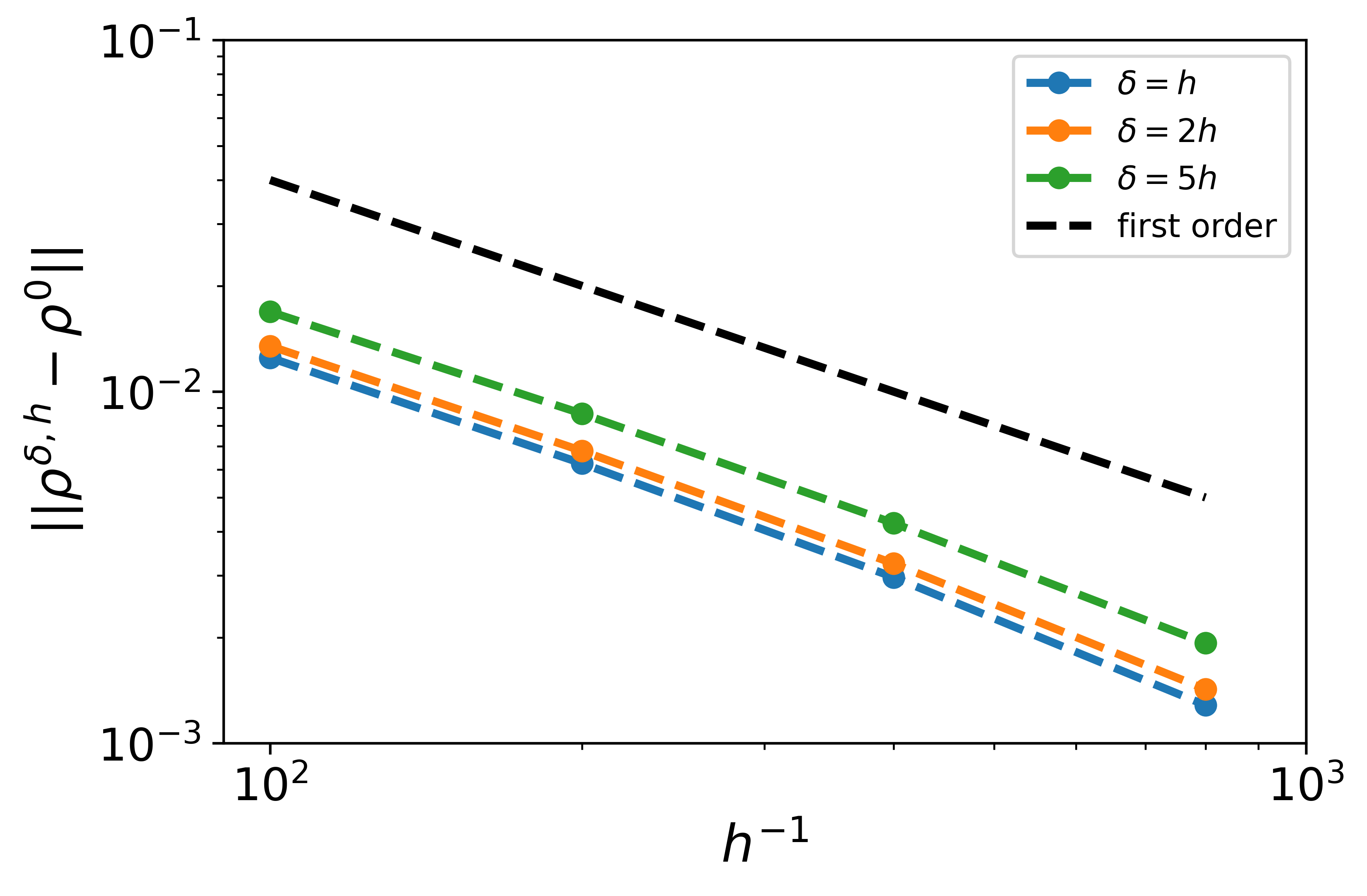}
	\end{subfigure}
	\begin{subfigure}{.32\textwidth}
	\includegraphics[width=\textwidth]{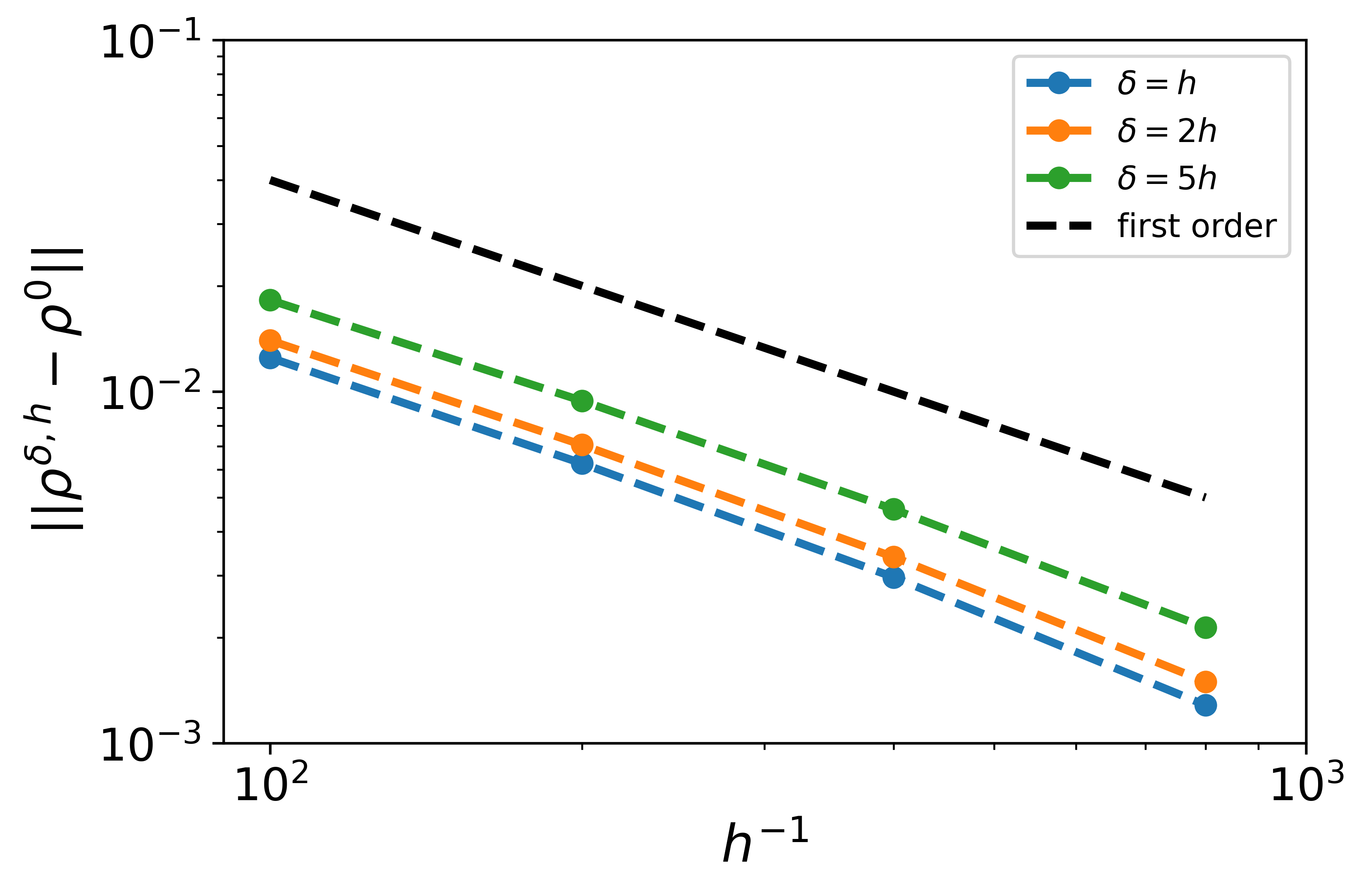}
	\end{subfigure}
	\begin{subfigure}{.32\textwidth}
	\includegraphics[width=\textwidth]{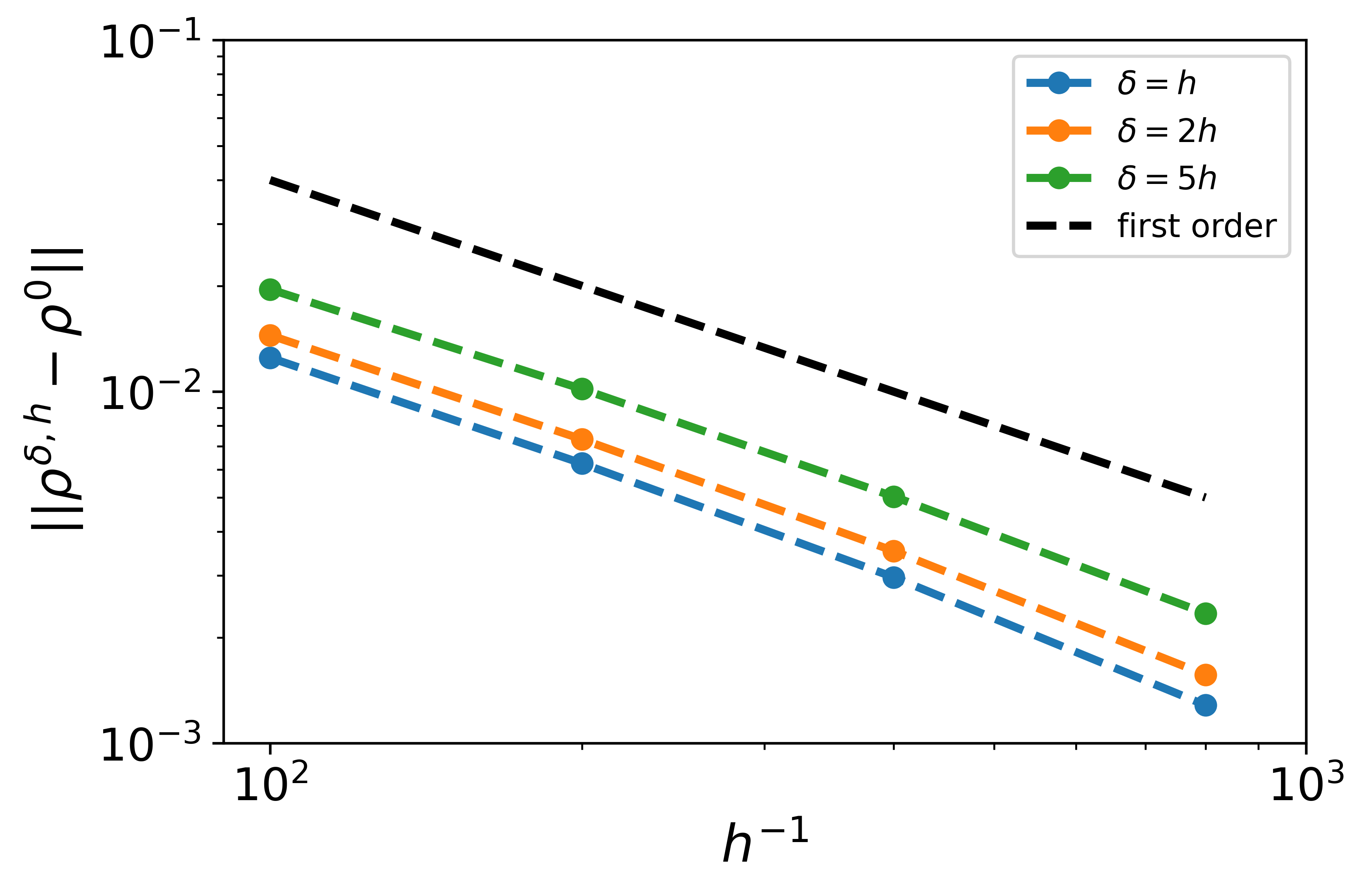}
	\end{subfigure}
	\caption{Experiment 4: Convergence from $\rho^{\delta,h}$ to $\rho^0$ for the bell-shaped initial data (\emph{top}) and the Riemann initial data (\emph{bottom}) corresponding to the linear decreasing kernel (\emph{left}), the exponential kernel (\emph{middle}), and the constant kernel (\emph{right}).}
	\label{fig:exp_4}
\end{figure}

We observe from Figure~\ref{fig:exp_4} that: for all the three nonlocal kernels, the error $\norm{\rho^{\delta,h} - \rho^0}_{\mathbf{L}^1}$ has a linear decay rate with respect to $h$ for both initial data and in all cases $\delta=mh$ for $m=1,2,5$.
Moreover, the plots for the three nonlocal kernels have little difference.
For the linear decreasing kernel and the exponential kernel, the convergence result validates the conclusion of Theorem~\ref{thm:ac}.
For the constant kernel, it does not satisfy the condition that $w=w_\delta(s)$ is strictly decreasing, and \eqref{eq:weight_monotone} does not hold because $w_{k-1}-w_k=0$ for all $k=1,\cdots,m-1$.
In this case, the analysis used in the proof of Theorem~\ref{thm:a_priori_estimates} cannot give the necessary estimates on numerical solutions but the numerical results show that the conclusion of Theorem~\ref{thm:ac} may still be true.

\section{Conclusions and future work}

In this work, finite volume numerical schemes \eqref{eq:nonlocal_lwr_num}-\eqref{eq:nonlocal_lwr_num_2} are studied for solving the nonlocal LWR model \eqref{eq:nonlocal_lwr} with a parameter $\delta$ that measures the range of information exchange. 
An important observation is that,
based on both numerical analysis and computational experiments, certain numerical quadrature weights that provide consistent approximations in the case of a given $\delta>0$ may lead to consistency between the scheme \eqref{eq:nonlocal_lwr_num}-\eqref{eq:nonlocal_lwr_num_2} and the local limit \eqref{eq:lwr} of the nonlocal model \eqref{eq:nonlocal_lwr} as $\delta\to 0$ and $h\to 0$. For properly selected numerical quadrature weights, we are able to prove, under reasonable assumptions that the numerical solutions of the nonlocal model converge to the continuum solution of the nonlocal model with a fixed $\delta>0$ as $h\to 0$, while they converge to the entropy solution of the local continuum model \eqref{eq:lwr} as $\delta\to0$ and $h\to 0$ simultaneously.
That is,  such schemes are asymptotically compatible with its local limit.
We are able to demonstrate that these asymptotically compatible schemes can offer robust numerical simulations under the changes in $\delta$ due to the uniform convergence when the values of $\delta$ are within a proper range.

Our established results are based on the a priori estimates on the numerical solutions as given in Theorem~\ref{thm:a_priori_estimates}, subject to assumptions alluded to above.
As shown in the computational experiments, the normalization condition for numerical quadrature weights is essential to the asymptotically compatibility of the scheme \eqref{eq:nonlocal_lwr_num}-\eqref{eq:nonlocal_lwr_num_2}.
The experiments also suggest that the results of this work may be extended to the cases with more general nonlocal kernels and numerical flux functions. 
It might also be possible to establish the results with more general velocity functions $v=v(\rho)$ other than the linear one $v(\rho)=1-\rho$ used here and also more general initial data that may have negative jumps. 
Furthermore, with the a priori bounds on the numerical solutions and known estimates on the exact solutions, it is possible to derive a priori error estimates subject to suitable conditions on the regularities of continuum solutions.
These questions along with further generalizations and applications of nonlocal traffic flow models will be subjects of future research.

\bibliographystyle{siam}
\bibliography{ref}

\begin{thebibliography}{10}

\bibitem{Aggarwal2015}
{\sc A.~Aggarwal, R.~M. Colombo, and P.~Goatin}, {\em Nonlocal systems of
  conservation laws in several space dimensions}, SIAM Journal on Numerical
  Analysis, 53 (2015), pp.~963--983.

\bibitem{Amorim2015}
{\sc P.~Amorim, R.~M. Colombo, and A.~Teixeira}, {\em On the numerical
  integration of scalar nonlocal conservation laws}, ESAIM: Mathematical
  Modelling and Numerical Analysis, 49 (2015), pp.~19--37.

\bibitem{Berthelin2019}
{\sc F.~Berthelin, , P.~Goatin, and and}, {\em Regularity results for the
  solutions of a non-local model of traffic flow}, Discrete {\&} Continuous
  Dynamical Systems - A, 39 (2019), pp.~3197--3213.

\bibitem{Betancourt2011}
{\sc F.~Betancourt, R.~B{\"u}rger, K.~H. Karlsen, and E.~M. Tory}, {\em On
  nonlocal conservation laws modelling sedimentation}, Nonlinearity, 24 (2011),
  p.~855.

\bibitem{Blandin2016}
{\sc S.~Blandin and P.~Goatin}, {\em Well-posedness of a conservation law with
  non-local flux arising in traffic flow modeling}, Numerische Mathematik, 132
  (2016), pp.~217--241.

\bibitem{brenier1988discrete}
{\sc Y.~Brenier and S.~Osher}, {\em The discrete one-sided lipschitz condition
  for convex scalar conservation laws}, SIAM Journal on Numerical Analysis, 25
  (1988), pp.~8--23.

\bibitem{bressan2020entropy}
{\sc A.~Bressan and W.~Shen}, {\em Entropy admissibility of the limit solution
  for a nonlocal model of traffic flow}, arXiv preprint arXiv:2011.05430,
  (2020).

\bibitem{bressan2019traffic}
\leavevmode\vrule height 2pt depth -1.6pt width 23pt, {\em On traffic flow with
  nonlocal flux: a relaxation representation}, Archive for Rational Mechanics
  and Analysis, 237 (2020), pp.~1213--1236.

\bibitem{burger2020non}
{\sc R.~B{\"u}rger, P.~Goatin, D.~Inzunza, and L.~M. Villada}, {\em A non-local
  pedestrian flow model accounting for anisotropic interactions and walking
  domain boundaries}, Mathematical biosciences and engineering, 17 (2020),
  pp.~5883--5906.

\bibitem{Chalons2018}
{\sc C.~Chalons, P.~Goatin, and L.~M. Villada}, {\em High-order numerical
  schemes for one-dimensional nonlocal conservation laws}, {SIAM} Journal on
  Scientific Computing, 40 (2018), pp.~A288--A305.

\bibitem{Chiarello2019}
{\sc F.~A. Chiarello, P.~Goatin, and E.~Rossi}, {\em Stability estimates for
  non-local scalar conservation laws}, Nonlinear Analysis: Real World
  Applications, 45 (2019), pp.~668--687.

\bibitem{coclite2020general}
{\sc G.~M. Coclite, J.-M. Coron, N.~De~Nitti, A.~Keimer, and L.~Pflug}, {\em A
  general result on the approximation of local conservation laws by nonlocal
  conservation laws: The singular limit problem for exponential kernels},
  Annales de l'Institut Henri Poincar{\'e} C,  (2022).

\bibitem{colombo2021role}
{\sc M.~Colombo, G.~Crippa, M.~Graff, and L.~V. Spinolo}, {\em On the role of
  numerical viscosity in the study of the local limit of nonlocal conservation
  laws}, ESAIM: Mathematical Modelling and Numerical Analysis, 55 (2021),
  pp.~2705--2723.

\bibitem{colombo2021local}
{\sc M.~Colombo, G.~Crippa, E.~Marconi, and L.~V. Spinolo}, {\em Local limit of
  nonlocal traffic models: convergence results and total variation blow-up},
  Annales de l'Institut Henri Poincar{\'e} C, Analyse non lin{\'e}aire, 38
  (2021), pp.~1653--1666.

\bibitem{colombo2022nonlocal}
\leavevmode\vrule height 2pt depth -1.6pt width 23pt, {\em Nonlocal traffic
  models with general kernels: singular limit, entropy admissibility, and
  convergence rate}, arXiv preprint arXiv:2206.03949,  (2022).

\bibitem{colombo2019singular}
{\sc M.~Colombo, G.~Crippa, and L.~V. Spinolo}, {\em On the singular local
  limit for conservation laws with nonlocal fluxes}, Archive for Rational
  Mechanics and Analysis, 233 (2019), pp.~1131--1167.

\bibitem{Colombo2012}
{\sc R.~M. Colombo, M.~Garavello, and M.~L{\'e}cureux-Mercier}, {\em A class of
  nonlocal models for pedestrian traffic}, Mathematical Models and Methods in
  Applied Sciences, 22 (2012), p.~1150023.

\bibitem{Colombo2018}
{\sc R.~M. Colombo and E.~Rossi}, {\em Nonlocal conservation laws in bounded
  domains}, SIAM Journal on Mathematical Analysis, 50 (2018), pp.~4041--4065.

\bibitem{evans2018measure}
{\sc L.~C. Evans and R.~F. Garzepy}, {\em Measure theory and fine properties of
  functions}, Routledge, 2018.

\bibitem{filbet2010class}
{\sc F.~Filbet and S.~Jin}, {\em A class of asymptotic-preserving schemes for
  kinetic equations and related problems with stiff sources}, Journal of
  Computational Physics, 229 (2010), pp.~7625--7648.

\bibitem{friedrich2022lyapunov}
{\sc J.~Friedrich, S.~G{\"o}ttlich, and M.~Herty}, {\em Lyapunov stabilization
  for nonlocal traffic flow models}, arXiv preprint arXiv:2209.05256,  (2022).

\bibitem{friedrich2022conservation}
{\sc J.~Friedrich, S.~G{\"o}ttlich, A.~Keimer, and L.~Pflug}, {\em Conservation
  laws with nonlocal velocity--the singular limit problem}, arXiv preprint
  arXiv:2210.12141,  (2022).

\bibitem{friedrich2019maximum}
{\sc J.~Friedrich and O.~Kolb}, {\em Maximum principle satisfying cweno schemes
  for nonlocal conservation laws}, SIAM Journal on Scientific Computing, 41
  (2019), pp.~A973--A988.

\bibitem{friedrich2018godunov}
{\sc J.~Friedrich, O.~Kolb, and S.~G{\"o}ttlich}, {\em A godunov type scheme
  for a class of lwr traffic flow models with non-local flux}, arXiv preprint
  arXiv:1802.07484,  (2018).

\bibitem{Goatin2019}
{\sc P.~Goatin and E.~Rossi}, {\em Well-posedness of {IBVP} for 1{D} scalar
  non-local conservation laws}, ZAMM-Journal of Applied Mathematics and
  Mechanics/Zeitschrift f{\"u}r Angewandte Mathematik und Mechanik, 99 (2019),
  p.~e201800318.

\bibitem{goatin2016well}
{\sc P.~Goatin and S.~Scialanga}, {\em Well-posedness and finite volume
  approximations of the {LWR} traffic flow model with non-local velocity},
  Networks and Hetereogeneous Media, 11 (2016), pp.~107--121.

\bibitem{Goettlich2014}
{\sc S.~G{\"o}ttlich, S.~Hoher, P.~Schindler, V.~Schleper, and A.~Verl}, {\em
  Modeling, simulation and validation of material flow on conveyor belts},
  Applied mathematical modelling, 38 (2014), pp.~3295--3313.

\bibitem{huang2022stability}
{\sc K.~Huang and Q.~Du}, {\em Stability of a nonlocal traffic flow model for
  connected vehicles}, SIAM Journal on Applied Mathematics, 82 (2022),
  pp.~221--243.

\bibitem{jin1999efficient}
{\sc S.~Jin}, {\em Efficient asymptotic-preserving (ap) schemes for some
  multiscale kinetic equations}, SIAM Journal on Scientific Computing, 21
  (1999), pp.~441--454.

\bibitem{jin2010asymptotic}
\leavevmode\vrule height 2pt depth -1.6pt width 23pt, {\em Asymptotic
  preserving (ap) schemes for multiscale kinetic and hyperbolic equations: a
  review}, Lecture notes for summer school on methods and models of kinetic
  theory (M\&MKT), Porto Ercole (Grosseto, Italy),  (2010), pp.~177--216.

\bibitem{karafyllis2020analysis}
{\sc I.~Karafyllis, D.~Theodosis, and M.~Papageorgiou}, {\em Analysis and
  control of a non-local {PDE} traffic flow model}, International Journal of
  Control, 0 (2020), pp.~1--34.

\bibitem{keimer2019approximation}
{\sc A.~Keimer and L.~Pflug}, {\em On approximation of local conservation laws
  by nonlocal conservation laws}, Journal of Mathematical Analysis and
  Applications, 475 (2019), pp.~1927--1955.

\bibitem{keimer2022singular}
\leavevmode\vrule height 2pt depth -1.6pt width 23pt, {\em On the singular
  limit problem for a discontinuous nonlocal conservation law}, arXiv preprint
  arXiv:2212.12598,  (2022).

\bibitem{lefloch2002hyperbolic}
{\sc P.~G. LeFloch}, {\em Hyperbolic Systems of Conservation Laws: The theory
  of classical and nonclassical shock waves}, Springer Science \& Business
  Media, 2002.

\bibitem{leveque2002finite}
{\sc R.~J. LeVeque et~al.}, {\em Finite volume methods for hyperbolic
  problems}, vol.~31, Cambridge university press, 2002.

\bibitem{lighthill1955kinematic}
{\sc M.~J. Lighthill and G.~B. Whitham}, {\em On kinematic waves {II}. {A}
  theory of traffic flow on long crowded roads}, Proc. R. Soc. Lond. A, 229
  (1955), pp.~317--345.

\bibitem{richards1956shock}
{\sc P.~I. Richards}, {\em Shock waves on the highway}, Operations research, 4
  (1956), pp.~42--51.

\bibitem{ridder2019traveling}
{\sc J.~Ridder and W.~Shen}, {\em Traveling waves for nonlocal models of
  traffic flow}, Discrete \& Continuous Dynamical Systems-A, 39 (2019),
  p.~4001.

\bibitem{rossi2020well}
{\sc E.~Rossi, J.~Wei{\ss}en, P.~Goatin, and S.~G{\"o}ttlich}, {\em
  Well-posedness of a non-local model for material flow on conveyor belts},
  ESAIM: Mathematical Modelling and Numerical Analysis, 54 (2020),
  pp.~679--704.

\bibitem{shen2018traveling}
{\sc W.~Shen}, {\em Traveling waves for conservation laws with nonlocal flux
  for traffic flow on rough roads}, Networks and Heterogeneous Media, 14
  (2019), pp.~709--732.

\bibitem{tadmor1984large}
{\sc E.~Tadmor}, {\em The large-time behavior of the scalar, genuinely
  nonlinear lax-friedrichs scheme}, Mathematics of computation, 43 (1984),
  pp.~353--368.

\bibitem{tian2014asymptotically}
{\sc X.~Tian and Q.~Du}, {\em Asymptotically compatible schemes and
  applications to robust discretization of nonlocal models}, SIAM Journal on
  Numerical Analysis, 52 (2014), pp.~1641--1665.

\bibitem{tian2020asymptotically}
\leavevmode\vrule height 2pt depth -1.6pt width 23pt, {\em Asymptotically
  compatible schemes for robust discretization of parametrized problems with
  applications to nonlocal models}, SIAM Review, 62 (2020), pp.~199--227.

\end{thebibliography}
\end{document}